\documentclass[a4paper, 10pt, twoside, notitlepage]{amsart}

\usepackage{amsmath,amscd}
\usepackage{amssymb}
\usepackage{amsthm}
\usepackage{comment}
\usepackage{graphicx}
\usepackage{epstopdf} 
\usepackage{mathrsfs}
\usepackage{xcolor}
\usepackage[ocgcolorlinks, linkcolor=blue]{hyperref}

\usepackage[ocgcolorlinks,linkcolor=blue]{hyperref}

\theoremstyle{plain}
\newtheorem{thm}{Theorem}[section]
\newtheorem{prop}{Proposition}[section]

\newtheorem{lem}[prop]{Lemma}

\newtheorem{rmk}[prop]{Remark}

\newcommand{\Om}{\Omega}
\numberwithin{equation}{section}
\newcommand {\R} {\mathbb{R}} 
 
\newcommand {\C} {\mathbb{C}} 
\newcommand {\p} {\partial}

\newcommand{\eps}{\epsilon}
\newcommand{\e}{\epsilon}
\newcommand{\dbar}{\ol{\p}}

\newcommand{\ol}[1]{\overline{#1}}
\newcommand{\oz}{\overline{z}}

\newcommand{\norm}[1]{\lVert #1 \rVert}

\newcommand{\abs}[1]{\lvert #1 \rvert}

\newcommand{\s}{\hspace{0.5pt}}

\DeclareMathOperator {\Imm} {Im}

\DeclareMathOperator{\Id} {Id}

\newcommand{\op}{\overline \p}

\newcommand{\kommentar}[1]{}

\title[An Inverse Problem for the Prescribed Mean Curvature]{An Inverse  Problem for the Prescribed Mean Curvature}

\author[T. Liimatainen]{Tony Liimatainen}
\address{Department of Mathematics and Statistics, University of Jyväskylä, Jyväskylä, Finland \& Department of Mathematics and Statistics, University of Helsinki, Helsinki, Finland}
\curraddr{}
\email{tony.t.liimatainen@jyu.fi, tony.liimatainen@helsinki.fi}

\author[J. Nurminen]{Janne Nurminen}
\address{Computational Engineering, School of Engineering Sciences, Lappeenranta-Lahti University of Technology, Finland \& Department of Mathematics and Statistics, University of Jyväskylä, Jyväskylä, Finland}
\curraddr{}
\email{janne.nurminen@lut.fi, janne.s.nurminen@jyu.fi}

\begin{document}
	\maketitle
	
	\begin{abstract}
We extend the recent study of inverse problems for minimal surfaces by considering the inverse source problem for the prescribed mean curvature equation
\begin{equation*}
    \nabla \cdot \left[ \frac{\nabla u}{(1 + |\nabla u|^2)^{1/2}} \right] = H(x).
\end{equation*}
This work also represents the first treatment of inverse source problems for quasilinear equations.
We prove that in two dimensions, the source function $H$ is uniquely determined by the associated Dirichlet-to-Neumann map. A notable feature of this problem is that although the equation is posed on an Euclidean domain, its linearization yields an anisotropic conductivity equation where the coefficient matrix corresponds to a Riemannian metric $g$ depending on the background solution. 

 The main methodological contribution is the derivation of a coupled nonlinear system of algebraic and geometric partial differential equations from boundary measurements. Similar systems will naturally appear in other inverse problems for quasilinear equations. We solve the system  using a Liouville type uniqueness result for conformal mappings, which recovers the source function uniquely.

		\medskip
		
		
		
	\end{abstract}

    \tableofcontents

\section{Introduction}
We consider an inverse problem for the prescribed mean curvature equation for $2$-dimensional surfaces living in $\R^3$. The equation in $\Omega \subset \R^n$ is 
\begin{equation}\label{PMC}
	\nabla\cdot\left[\frac{\nabla u}{(1+\abs{\nabla u}^2)^{1/2}}\right]=H(x),
\end{equation}
where $H\in C^\alpha(\Omega)$. 
Let us assume for now that the Dirichlet problem for \eqref{PMC} is well-posed on an open subset $\mathcal{N}\subset C^{2,\alpha}(\p \Omega)$.
In this case, the \emph{Dirichlet-to-Neumann map} (DN map) is defined by the usual assignment
   \begin{equation}\label{eq:DNmap}
   \Lambda_{H}: \mathcal{N}\to C^{1,\alpha}(\p \Omega), \qquad f\mapsto \left. \p_\nu u_f\right|_{\p \Omega},
   \end{equation}
   where $u_f$ is the unique (small) solution to $\eqref{PMC}$ with boundary value $f$.
     Here $\nu$ denotes the unit outer normal on $\p \Omega$. 
     
     The well-posedness holds if $H\in C^1(\Omega)$ and is small enough when compared to the isoperimetric constant and mean curvature of the boundary of $\Omega$, see \cite[Theorem 16.10]{gilbarg2001elliptic}. By the same theorem for some
     \begin{equation*}
         f_0\in C^{2,\alpha}(\p \Omega)\text{ there is a unique solution }u_{0}\in C^{2,\alpha}(\ol{\Om}).
     \end{equation*}
     Now we can make a small perturbation by $f\in C^{2,\alpha}(\p \Omega)$ to the boundary value $f_0$ to obtain a unique solution $u_f$ of \eqref{PMC} with $u_f|_{\p\Om}=f_0+f$. Then we can use the implicit function theorem to say that this solution depends smoothly on $f$. Notice that by \cite[Theorem 16.10]{gilbarg2001elliptic} we do not need to restrict ourselves to unique solutions in some neighborhood of $u_0$, sometimes called small solutions in previous works. We refer to  \cite[Theorem 2.1]{liimatainenLin2024} for a related well posedness proof. 
     

We ask the following question:
\begin{center}
	\textbf{Inverse problem:} Can we determine $H$ from the DN map?
\end{center}
This type of an inverse problem is called an inverse source problem. We will give a positive answer to this:

\begin{thm}\label{main_thm}
	Let $H,\tilde H\in C^{\infty}(\ol{\Om})$ be such that $H|_{\p\Om}=\tilde H|_{\p\Om}$ to high order and $f_0\in C^{\infty}(\p\Om)$. Assume that $u_0$ and $\tilde u_0$ solve 
    \begin{equation}\label{main_pde}
 	\left\{\begin{array}{ll}
 		\nabla\cdot\left[\frac{\nabla u}{(1+\abs{\nabla u}^2)^{1/2}}\right]=H(x), & \text{in}\,\, \Om \\
 		u=f_0, & \text{on}\,\, \partial\Om,
 	\end{array} \right.
    \end{equation}
     for $H$ and $\tilde H$ on the right hand side respectively, and with $u_0|_{\p\Om}=\tilde u_0|_{\p\Om}= f_0$. Suppose that the DN maps of \eqref{PMC} satisfy 
	\begin{equation*}
		\Lambda_{H}f=\Lambda_{\tilde H}f,\quad\text{for all } f\in C^{\infty}(\p\Om)
	\end{equation*}
	with $\norm{f_0 - f}_{C^{\infty}(\p\Om)}\leq\delta$ for sufficiently small $\delta>0$. Then $H=\tilde H$ in $\Om$.
\end{thm}

\subsection{Motivation}
The prescribed mean curvature (PMC) equation is a prototypical quasilinear elliptic equation of the form
\begin{equation}\label{eq:PMC}
	\nabla \cdot \left( A(\nabla u) \nabla u \right) = H(x),
\end{equation}
where \( A : \R^n \to \R \). The main motivation of the present work is to serve as a first step in solving inverse problems for equations of the above. To the best of our knowledge, it is the first to consider an inverse source problem for a quasilinear elliptic equation.

Another key motivation for our study concerns the method we employ to solve the inverse problem. Our approach is based on the higher order linearization technique. The first step of this method involves solving the Calder\'on problem for the first linearized equation
\begin{equation}\label{eq:linearized}
	\begin{cases}
		\nabla \cdot \left( g^{-1} \nabla v \right) = 0, & \text{in } \Om, \\
		v = f, & \text{on } \partial \Om,
	\end{cases}
\end{equation}
where \( g \) is a positive definite symmetric matrix field, thus a Riemannian metric. This linearized equation is an anisotropic conductivity equation. Therefore, our study provides a new application for the theory of the anisotropic Calder\'on problem, even though the original nonlinear equation is posed on a Euclidean domain.

Regarding the PMC equation itself, it generalizes the minimal surface equation (\( H=0 \)) by requiring a surface to have a specific, pre-assigned mean curvature \( H \) at every point, rather than zero curvature. Physically, this corresponds to modeling surfaces in a state of balanced tension under external forces, such as soap bubbles with internal pressure, rather than area-minimizing soap films. Consequently, the PMC equation serves as a fundamental model in diverse areas, including capillary surfaces, biological membranes, and general relativity, where it arises in the initial data formulation of Einstein's equations \cite{seifert1997configurations, chrusciel2010mathematical}.

Beyond these classical applications, we mention that inverse problems for minimal surfaces have recently gained surprising significance  within the AdS/CFT duality theory in theoretical physics. There, techniques from inverse problems have proven effective in such cases \cite{jokela2025bulk}  where the mathematical
tools traditionally used in physics have been ineffective. 
Due to this surprising application, we mention the works \cite{hung2011holographic, dong2014holographic} that study surfaces with prescribed curvature in holography.

\subsection{Earlier works}

Our work is a new addition to a growing branch of inverse problems for nonlinear partial differential equations. In \cite{isakov1993uniqueness} the author linearized the DN map to solve an inverse problem for a nonlinear PDE. By linearizing one can then use methods for linear equations to solve the problem. In recent years the interest for these has grown and it started with the works \cite{kurylev2018inverse,FO19,LLLS2019inverse}. In these works the method of higher order linearization was introduced. In this method one linearizes the DN map multiple times and solves an inverse problem at each step, gaining new information about the original problem.

This method has been applied to many problems. Quasilinear elliptic isotropic conductivity equations have been considered in \cite{CFKKU, kian2020partial, liimatainen2023calder}. In these works the goal is to recover the nonlinear conductivity. Also inverse problems for the minimal surface equation, quasilinear in nature, on Riemannian manifolds and Euclidean domains have been considered in \cite{ABN_20_minimal,carstea2024calderonproblemriemanniansurface,carstea2023inverseproblemgeneralminimal,carstea2024calderonproblemriemanniansurfaces,nurminen2023,Nurminen2024}, and in the context of AdS/CFT duality in physics in \cite{jokela2025bulk}. Also in \cite{muñozthon2024calderonsproblemharmonicmaps} the author studies a Calderón type problem for harmonic maps. A map between two Riemannian manifolds is harmonic if it satisfies a certain quasilinear equation.

Inverse problems for elliptic second order semilinear equations have been considered in \cite{isakov1994global, victorN, sun2010inverse, imanuvilovyamamoto_semilinear, LLLS2019inverse, LLLS2021b, liimatainen2022inverse, MR4052205,FO19, FLL23, Salo2022, Nurminen2023single, johansson2023inverseproblemssemilinearelliptic, johansson2025inverseproblemssemilinearelliptic} and inverse problems for nonlinear magnetic Schrödinger equations in \cite{Krupchyk2020, Lai2023partial, ma2020note, Krupchyk2024remark}. We also mention the works for nonlinear biharmonic equations \cite{bhatta2025, nurminen2025inverseproblemnonlinearbiharmonic}, $p$-harmonic equation \cite{MR3320025, MR3810150, carstea2025two} and the works for nonlinear fractional equations \cite{Lin2022, lai2019global, LinLiu2023, LaiLin2019}. See also \cite{lassas2025introductioninverseproblemsnonlinear} for a recent survey on this subject.

The above mentioned works consider PDEs with zero right hand side, that is, they do not have a source term. Recently in \cite{liimatainenLin2024,kianliimatainenlin2024} the authors considered inverse problems for nonlinear partial differential equations where the aim is to recover a nonlinear term and a source term. They also demonstrate by example nonlinearities that one can break a gauge appearing in a general case of recovery.

\subsection{Method of proof}
The methods we use can be divided into two main parts. First, we combine information from the Dirichlet-to-Neumann (DN) map with the higher-order linearization technique and complex geometric optics (CGO) solutions for the Schrödinger equation to derive equations for quantities related to the background solution $u_0$ (corresponding to the boundary value $f_0$). These equations are highly coupled, and the second step involves solving for $\nabla u_0$ from this system. A key ingredient at this stage is a Liouville-type uniqueness result for conformal mappings.

To show that $H = \tilde{H}$ in $\Omega$, we proceed as follows. If $u_0$ and $\tilde{u}_0$ are the solutions of the prescribed mean curvature equation corresponding to the sources $H$ and $\tilde{H}$ respectively, and if we establish that $\nabla u_0 = \nabla \tilde{u}_0$, then it follows immediately from the structure of the equation that
\begin{equation*}
    H = \nabla \cdot \left[ \frac{\nabla u_0}{(1 + |\nabla u_0|^2)^{1/2}} \right]
      = \nabla \cdot \left[ \frac{\nabla \tilde{u}_0}{(1 + |\nabla \tilde{u}_0|^2)^{1/2}} \right]
      = \tilde{H}.
\end{equation*}

We now provide a detailed outline of the argument. The first linearization of \eqref{PMC} yields an anisotropic conductivity equation, equivalent to a Schrödinger equation involving the Laplace–Beltrami operator associated with the Riemannian metric whose inverse is
\begin{equation*}
    g^{-1} = \frac{1}{(1 + |\nabla u_0|^2)^{1/2}} \left( I_{2 \times 2} - \frac{\nabla u_0 \otimes \nabla u_0}{1 + |\nabla u_0|^2} \right).
\end{equation*}
This metric explicitly depends on the background solution $u_0$. Using boundary information about $H$, we show that the DN maps for the corresponding Schrödinger equations coincide. Applying \cite[Theorem 1.1]{carstea2024calderonproblemriemanniansurfaces}, we conclude that the metrics associated with $u_0$ and $\tilde{u}_0$, along with their potentials, agree up to natural conformal and diffeomorphism gauges; see Section \ref{sec_first_lin}.

The next step is to break these gauge invariances. Using unique continuation and the specific functional dependence of the potential on the metric $g$, we deduce (see Lemma \ref{J_lambda} and \eqref{norm_u_tilde_u}) that
\begin{equation}\label{eq:norm_nabla_u0}
    |\nabla u_0| = |\nabla \tilde{u}_0| \circ \phi,
\end{equation}
where $\phi: \Omega \to \Omega$ is a diffeomorphism that fixes the boundary.

To strengthen this to full gradient equality, we employ the second linearization, which yields the integral identity:
\begin{align}\label{eq:integral_identity}
    0 = \int_{\Omega} T \cdot \bigg( & \nabla v_0 \, g(\nabla v_1, \nabla v_2) \\
    & + \nabla v_1 \left( \Delta_g(v_0 v_2) - v_2 \Delta_g v_0 - v_0 \Delta_g v_2 \right) \notag \\
    & + \nabla v_2 \left( \Delta_g(v_1 v_0) - v_0 \Delta_g v_1 - v_1 \Delta_g v_0 \right) \bigg)  dV_g, \notag
\end{align}
where $T$ is a
\begin{equation*}
T:=\left(\frac{\lambda J_{\phi} \phi^*\nabla \tilde u_0}{1+\abs{\nabla \tilde u_0}^2|_{\phi}} - \frac{\nabla u_0}{1+\abs{\nabla u_0}^2}\right)\abs{g}^{-1/2},
\end{equation*}
and $v_k$ ($k=0,1,2$) are solutions to the first linearized equation. The analysis of this identity relies on carefully constructed CGO solutions \cite{guillarmou2011identification}, where we choose two solutions without critical points and one with a critical point. Using asymptotic analysis techniques similar to \cite{carstea2024calderonproblemriemanniansurfaces, liimatainen2023calder}, we ultimately conclude that (see \eqref{relation_u_tilde_u_3})
\begin{equation}\label{eq:gradient_equality}
    \nabla u_0 = \phi^* \nabla \tilde{u}_0.
\end{equation}

The final step is to show that \eqref{eq:norm_nabla_u0} and \eqref{eq:gradient_equality} together imply that $\phi$ is the identity map on $\Omega$. This follows from a uniqueness theorem for conformal mappings, which forces $\phi = \text{id}$, see Section \ref{sub_sec_phi_identity}. Consequently, $\nabla u_0 = \nabla \tilde{u}_0$ throughout $\Omega$, and thus $H = \tilde{H}$ as required.

\section{First linearization}\label{sec_first_lin}

To make notation more convenient, we let $F\colon \R^2\to\R^2,$
\begin{equation*}
	F(p):=\frac{p}{(1+\abs{p}^2)^{1/2}}
\end{equation*}
and thus \eqref{PMC} can be written as
\begin{equation}\label{eq:F_H_equation}
	\nabla\cdot F(\nabla u) = H \text{ in } \Om.
\end{equation}
We assume that there exists a solution $u_0\in C^{2,\alpha}(\Omega)$ to \eqref{eq:F_H_equation} corresponding to boundary value $f_0\in C^{2,\alpha}(\p\Omega)$ and then we linearize at this solution. Let $f_1,f_2\in C^{\infty}(\p\Om)$, $\e=(\e_1,\e_2)$ and $\e_1, \e_2>0,$ be such that $f=f_0 + \e_1f_1 + \e_2f_2$ is sufficiently close to $f_0$. Now the first linearization, denoted by 
\[
v_1=\p_{\e_1}u|_{\e=0},
\]
satisfies the first linearized equation
\begin{equation}\label{first_lin}
	\p_a(\p_bF^a(\nabla u)|_{u=u_0}\p_{\e}\p_bu|_{\e=0}) = \p_a(\p_bF^a(\nabla u_0)\p_b v) = 0 \text{ in }\Om,
\end{equation}
where
\begin{equation}\label{first_deriv_F}
	\p_bF^a(p) = \frac{\delta^{ab}}{(1+\abs{p}^2)^{1/2}} - \frac{p_ap_b}{(1+\abs{p}^2)^{3/2}}.
\end{equation}
From now on, we choose to put the indexes up and define the Riemannian metric
\begin{equation}\label{def_g_0}
	g^{ab}=\frac{1}{(1+\abs{\nabla u_0}^2)^{1/2}}\left(\delta^{ab} - \frac{\p^au_0\p^bu_0}{1+\abs{\nabla u_0}^2}\right)=\p_bF^a(\nabla u_0).
\end{equation}
For a proof that $g^{ab}$ is indeed positive definite symmetric matrix we refer to \cite[Chapter 10]{gilbarg2001elliptic}. Thus the first linearization is the elliptic equation
\begin{equation}\label{first_lin_vol0}
	\p_a(g^{ab}\p_bv)=0 \text{ in }\Om.
\end{equation}
For boundary values $f$ close to $f_0$, the boundary value problem
\begin{equation*}
	\left\{\begin{array}{ll}
		\nabla\cdot F(\nabla u)=H, & \text{in}\,\, \Om \\
		u=f_0 + f, & \text{on}\,\, \partial\Om
	\end{array} \right.
\end{equation*}
is well-posed and  the DN map is smooth in the Fréchet sense:
\begin{prop}
	Let $n\geq2$, $\Om\subset\R^n$ be a bounded domain with smooth boundary, $H\in C^{\infty}(\ol{\Om})$ and $f_0\in C^{\infty}(\p\Om)$. Assume that $u_0\in C^{\infty}(\ol{\Om})$ is a solution to
	\begin{equation*}
		\left\{\begin{array}{ll}
			\nabla\cdot F(\nabla u_0)=H, & \text{in}\,\, \Om \\
			u_0=f_0, & \text{on}\,\, \partial\Om.
		\end{array} \right.
	\end{equation*}
	Then there are $\delta, C>0$ such that for any
	\begin{equation*}
		f\in B_{\delta}=\{f\in C^{\infty}(\ol{\Om}) : \norm{f_0 - f}_{C^{\infty}(\ol{\Om})}\leq\delta\}
	\end{equation*}
	there exists a unique small solution $u\in C^{\infty}(\ol{\Om})$ of
	\begin{equation*}
		\left\{\begin{array}{ll}
			\nabla\cdot F(\nabla u)=H, & \text{in}\,\, \Om \\
			u=f_0 + f, & \text{on}\,\, \partial\Om
		\end{array} \right.
	\end{equation*}
	so that $\norm{u_0 - u}_{C^{\infty}(\ol{\Om})}\leq C\delta$. Furthermore there are $C^{\infty}$ maps, in the Fréchet sense,
	\begin{align*}
		S&\colon B_{\delta}\to C^{\infty}(\ol{\Om}), \quad f\mapsto u\\
		\Lambda_H&\colon B_{\delta}\to C^{\infty}(\ol{\Om}), \quad f\mapsto \p_{\nu}u|_{\p\Om}.
	\end{align*}
\end{prop}

The proof of this is similar to the one in \cite{liimatainenLin2024} and we omit the proof. We note that the linearized equation \eqref{first_lin_vol0} is well-posed (since it does not contain a zeroth order term) and thus we do not need to make assumptions regarding the well-posedness of the linearization of \eqref{PMC}. If $H$ is small enough compared to the isoperimetric constant of $\Omega$ and the mean curvature of the boundary $\p \Omega$, then $u_0$ exists and the solution $u$ corresponding to boundary value $f+f_0$ is globally unique. For these facts, see \cite[Theorem 16.10, and (16.61)]{gilbarg2001elliptic}.

We modify \eqref{first_lin} to a Laplace-Beltrami type equation. 
For $\gamma=\abs{g}^{-1/2}$ we have
\begin{equation*}
	0=\nabla\cdot\left(g^{-1}\nabla(\gamma^{-1/2}v)\right) = \gamma^{-1/2}(\Delta_g + q)v
\end{equation*}
where $q=\dfrac{\Delta_g(\gamma^{1/2})}{\gamma^{1/2}}$. Thus $v$ is a solution of the boundary value problem 
\begin{equation}\label{first_lin_divergence}
	\left\{\begin{array}{ll}
		\nabla\cdot\left(g^{-1}\nabla v\right)=0, & \text{in}\,\, \Om \\
		v=f, & \text{on}\,\, \partial\Om
	\end{array} \right.
\end{equation}
if and only if $\hat{v}=\gamma^{1/2}v$ is a solution to
\begin{equation}\label{first_lin_schrödinger}
	\left\{\begin{array}{ll}
		(\Delta_g + q)\hat{v}=0, & \text{in}\,\, \Om \\
		\hat{v}=\gamma^{1/2}f, & \text{on}\,\, \partial\Om.
	\end{array} \right.
\end{equation}


If the DN maps corresponding to two sources $H$ and  $\tilde H$ agree, it follows by linearization that the DN maps of the equation \eqref{first_lin_divergence} for the coefficients $(g,q)$ and $(\tilde g, \tilde q)$ agree. In this case there exists a conformal diffeomorphism $\phi\colon \ol{\Om}\to\ol{\Om}$ such that
\begin{equation}\label{first_lin_determination}
	g=\lambda\phi^*\tilde g\quad \text{and that}\quad q=\lambda^{-1}\phi^*\tilde q
\end{equation}
with $\lambda|_{\p\Om}=1$ and $\phi|_{\p\Om}= \Id$, see \cite{carstea2024calderonproblemriemanniansurfaces}.

Let us then denote by $v$ and $\tilde v$ the solutions to the first linearizations $(\Delta_{g}+q)v=0$ and  $(\Delta_{\tilde g}+\tilde q)\tilde v=0$ corresponding to a given boundary value $f\in C^{2,\alpha}(\p\Omega)$. 
Now 
\begin{align}\label{solutions_linearized}
	&(\Delta_{g} + q)(\tilde{v}\circ\phi) = (\Delta_{\lambda\phi^*\tilde g} + \lambda^{-1}\phi^*\tilde q)(\tilde{v}\circ\phi) \\\notag
	&= \lambda^{-1}(\Delta_{\phi^*\tilde g} + \phi^*\tilde q)(\tilde{v}\circ\phi) = \lambda^{-1}\phi^*(\Delta_{\tilde g}\tilde v + \tilde q\tilde v) = 0.
\end{align}
Since $\phi|_{\p\Om}=\mathrm{Id}$, we have $v=\tilde{v}\circ\phi$ by the uniqueness of solutions to the Dirichlet problem.

Using the Sylvester determinant rule and \eqref{def_g_0} we have the following
\begin{align*}
	\det(\lambda^{-1}(D\phi)^{-1}\tilde g^{-1}(D\phi)^{-T}) &= \lambda^{-2}J_{\phi}^{-2} \frac{1}{1+\abs{\nabla\tilde u_0}^2|_{\phi}}\left(1-\frac{\abs{\nabla\tilde u_0}^2|_{\phi}}{1+\abs{\nabla\tilde u_0}^2|_{\phi}}\right)\\
	&= \lambda^{-2}J_{\phi}^{-2}\frac{1}{(1+\abs{\nabla\tilde u_0}^2|_{\phi})^2}.
    \end{align*}
On the other hand
    \begin{align*}
	 \det(g^{-1})&=\frac{1}{1+\abs{\nabla u_0}^2}\left(1-\frac{\abs{\nabla u_0}^2}{1+\abs{\nabla u_0}^2}\right)\\
	&=\frac{1}{(1+\abs{\nabla u_0}^2)^2}.
\end{align*}
Thus, the identity
\[
g^{-1}=\lambda^{-1}(D\phi)^{-1}\tilde g^{-1}|_{\phi}(D\phi)^{-T}
\]
gives
\begin{equation*}
	(\lambda J_{\phi})^{-2}=\frac{(1+\abs{\nabla\tilde u_0}^2|_{\phi})^2}{(1+\abs{\nabla u_0}^2)^2},
\end{equation*}
or equivalently
\begin{equation}\label{lambda_times_J}
	\lambda J_{\phi} = \frac{1+\abs{\nabla u_0}^2}{1+\abs{\nabla\tilde u_0}^2|_{\phi}}.
\end{equation}
For the next result we recall that
\begin{equation}\label{eq:formula_for_q}
q=\dfrac{\Delta_g(\gamma^{1/2})}{\gamma^{1/2}} \text{ and } \tilde q=\dfrac{\Delta_{\tilde g}(\tilde \gamma^{1/2})}{\tilde \gamma^{1/2}},
\end{equation}
where
\[
\gamma=\abs{g}^{-1/2} \text{ and }\tilde \gamma=\abs{\tilde g}^{-1/2}.
\]

\begin{lem}\label{J_lambda}
	Let $\Om\subset\R^2$ be a domain with smooth boundary, $g, \tilde g$ be Riemannian metrics on $\Om$, $\phi\colon \Om\to\Om$ a diffeomorphism such that $\phi|_{\p\Om}=\Id$ and $\lambda\colon\Om\to (0,\infty)$ such that $\lambda|_{\p\Om}=1$. Assume that $g=\lambda\phi^*\tilde g$,
	\begin{equation}\label{boundary_eq}
		\det(g^{-1}\tilde g|_\phi)=1\quad\text{to first order on } \p \Omega 
	\end{equation}
	and 
	\begin{equation}\label{q_identity}
		q=\lambda^{-1}\phi^*\tilde q,
	\end{equation}
    where $q$ and $\tilde q$ are as in \eqref{eq:formula_for_q}. 
	Then
	\begin{equation}\label{J_lambda_identity}
		J_{\phi}\lambda=1.
	\end{equation}
\end{lem}

\begin{proof}
	By using the conformal invariance of the Laplacian, $\Delta_{cg} v=c^{-1}\Delta_gv$,  together with \eqref{q_identity} and $g=\lambda\phi^*\tilde g$ we have that
	\begin{align*}
		q=\abs{g}^{1/4}\Delta_g(\abs{g}^{-1/4}) &=\lambda^{-1}\phi^*\tilde q = \lambda^{-1}\phi^*\left(\abs{\tilde g}^{1/4}\Delta_{\tilde g}(\abs{\tilde g}^{-1/4})\right)\\
        &=\lambda^{-1}\abs{\tilde g}^{1/4}|_{\phi}\Delta_{\phi^*\tilde g}(\abs{\tilde g}^{-1/4}|_{\phi})\\
		&= \lambda^{-1}\abs{\tilde g}^{1/4}|_{\phi}\Delta_{\lambda^{-1}g}(\abs{\tilde g}^{-1/4}|_{\phi})\\
		&= \abs{\tilde g}^{1/4}|_{\phi}\Delta_{g}(\abs{\tilde g}^{-1/4}|_{\phi}).
	\end{align*}
	Let us denote $D=\abs{g}^{-1/4}$ and $\tilde D=\abs{\tilde g}^{1/4}|_{\phi}$. Then the previous equality becomes
	\begin{align*}
		0=\frac{\Delta_gD}{D} - \frac{\Delta_g\tilde D}{\tilde D} = \frac{1}{\tilde D}\Delta_g(D-\tilde D) + \Delta_gD\left(\frac{1}{D} - \frac{1}{\tilde D}\right).
	\end{align*}
	This implies that
	\begin{align*}
		\abs{\Delta_g(D-\tilde D)}=\abs{\tilde D \Delta_gD}\abs{\frac{D-\tilde D}{D\tilde D}} \leq C \abs{D-\tilde D}
	\end{align*}
	for some positive constant $C$. By \eqref{boundary_eq}, we have that $D$ and $\tilde D$ agree to first order on $\p\Omega$. It then follows by the standard unique continuation for the Laplace equation (see e.g \cite[Theorem B.1.]{Kenig2011}) that
	\begin{equation*}
		D=\tilde D\ \text{ or equivalently } \ \abs{g}=\abs{\tilde g}|_{\phi}.
	\end{equation*}
	By the assumption $g=\lambda\phi^*\tilde g$ and $\abs{g}=\abs{\tilde g}|_{\phi}$, we obtain 
	\begin{equation*}
		\abs{g} = \abs{\lambda\phi^*\tilde g} = \lambda^2J_{\phi}^2\abs{\tilde g}|_{\phi} = \lambda^2J_{\phi}^2\abs{g}
	\end{equation*}
	which then implies \eqref{J_lambda_identity}.
\end{proof}

Now from \eqref{J_lambda_identity} and \eqref{lambda_times_J} we get
\begin{equation}\label{norm_u_tilde_u}
	\abs{\nabla u_0}=\abs{\nabla\tilde u_0}|_{\phi}.
\end{equation}

\begin{rmk}

We note that prescribed Jacobian equations such as \eqref{J_lambda_identity} may lack unique solutions \cite[Section 1.4]{csato2011pullback}. To resolve this ambiguity, we proceed to the second linearization in Section \ref{sec_second_lin}, which yields additional constraints on $\phi$. Ultimately, these constraints force $\phi$ to be a conformal mapping, whose uniqueness then follows from Liouville-type rigidity.

\end{rmk}

\section{Boundary information}\label{sec_boundary}

Before moving to the second linearization in this section we discuss the needed information on the boundary. As seen in the previous section, the metric $g$ introduced in \eqref{def_g_0} depends on the initial solution $u_0$. Thus in order to have knowledge of $g$ on the boundary, we need to have knowledge of $u_0$ on the boundary.

First of all, since we know $u_0$ and $\p_{\nu}u_0$ on the boundary, we know $\nabla u_0$ on the boundary. Thus we know $g$ on the boundary to zeroth order.

In Section \ref{sec_T=0} we will need that the conformal diffeomorphism $\phi$ from the previous section is the identity to high order on the boundary which follows from \cite[Lemma 5.1]{carstea2023inverseproblemgeneralminimal}. We only need to check that $\phi$ satisfies the properties required by that Lemma. From \eqref{solutions_linearized} we know that $\phi$ is a morphism of solutions. We also know that $\Lambda_H = \Lambda_{\tilde H}$ and combining this with the fact that $g, \tilde{g}$ agree on the boundary further implies that the DN maps for the first linearizations agree.

Thus in order to use \cite[Lemma 5.1]{carstea2023inverseproblemgeneralminimal} we need to have that $g, \tilde g$ agree on the boundary in all orders. This follows, if we can show that $u_0$ and $\tilde u_0$ agree on the boundary in all orders.

We know $H|_{\p\Om}=\tilde H|_{\p\Om}$ and thus 
\begin{equation}\label{eq_boundary_H}
	\nabla\cdot\left(\frac{\nabla u_0}{(1+\abs{\nabla u_0}^2)^{1/2}}\right)\bigg|_{\p\Om} = 	\nabla\cdot\left(\frac{\nabla \tilde u_0}{(1+\abs{\nabla \tilde u_0}^2)^{1/2}}\right)\bigg|_{\p\Om}.
\end{equation}
Now
\begin{align*}
	\nabla\cdot\left(\frac{\nabla u_0}{(1+\abs{\nabla u_0}^2)^{1/2}}\right) = \frac{\Delta u_0}{(1+\abs{\nabla u_0}^2)^{1/2}} - \frac{\nabla^2u_0(\nabla u_0,\nabla u_0)}{(1+\abs{\nabla u_0}^2)^{3/2}}
\end{align*}
which can be divided into terms containing tangential derivatives of $u_0$, which we already know on $\p\Om$, and terms containing only normal derivatives. For the moment let us assume that the normal direction is in the direction $x_2$. The ones with only normal derivatives are
\begin{align*}
	\frac{\p^2_2u_0}{(1+\abs{\nabla u_0}^2)^{1/2}} - \frac{\p^2_2 u_0\p_2u_0\p_2u_0}{(1+\abs{\nabla u_0}^2)^{3/2}} = \p^2_2u_0\left(\frac{1}{(1+\abs{\nabla u_0}^2)^{1/2}} - \frac{(\p_2u_0)^2}{(1+\abs{\nabla u_0}^2)^{3/2}}\right).
\end{align*}
Then from \eqref{eq_boundary_H} it follows
\begin{equation*}
    \frac{\p^2_2u_0}{(1+\abs{\nabla u_0}^2)^{1/2}}\left(1-\frac{(\p_2u_0)^2}{1+\abs{\nabla u_0}^2}\right) = \frac{\p^2_2\tilde u_0}{(1+\abs{\nabla \tilde u_0}^2)^{1/2}}\left(1-\frac{(\p_2\tilde u_0)^2}{1+\abs{\nabla \tilde u_0}^2}\right)
\end{equation*}
and using $\nabla u_0=\nabla\tilde u_0$ on $\p\Om$ implies
\begin{equation*}
    \left(\p^2_2u_0 - \p^2_2\tilde u_0\right)\left(1-\frac{(\p_2u_2)^2}{1+\abs{\nabla u_0}^2}\right) = 0.
\end{equation*}
This is true when $\p^2_2u_0 = \p^2_2\tilde u_0$ since otherwise we would have $(\p_2u_0)^2=-1$.
Thus from $H|_{\p\Om}=\tilde H|_{\p\Om}$ it follows $\nabla^2 u_0|_{\p\Om} = \nabla^2\tilde u_0|_{\p\Om}$.

For higher order derivatives of $u_0$ we again only need to focus on terms with only higher order normal derivatives. Since we know $\p_2H|_{\p\Om}=\p_2\tilde H|_{\p\Om}$, we compute $\p_2 H$:
\begin{align*}
	\p_2H &= \frac{\p_2\Delta u_0}{(1+\abs{\nabla u_0}^2)^{1/2}} \\
	&- (1+\abs{\nabla u_0}^2)^{-3/2}\left(\p^3_{211}u_0(\p_1 u_0)^2 + \p^3_{212}u_0\p_1u_0\p_2u_0 + \p^2_{221}u_0\p_1u_0\p_2u_0 + \p^3_2u_0(\p_2u_0)^2\right)\\
	&+ R(\nabla^2u_0,\nabla u_0)
\end{align*}
where the term $R$ does not include third order derivatives. Again we are interested in the term with only $\p_2$ derivatives in them:
\begin{align*}
	\p^3_2u_0\left(\frac{1}{(1+\abs{\nabla u_0}^2)^{1/2}} - \frac{(\p_2u_0)^2}{(1+\abs{\nabla u_0}^2)^{3/2}}\right).
\end{align*}
This is similar to situation above and thus we would have $\p^3_2u_0 = \p^3_2\tilde u_0$.
Hence we get that the third order derivatives of $u_0$ and $\tilde u_0$ agree on $\p\Om$. Proceeding inductively we obtain that $u_0$ and $\tilde u_0$ agree to high order on $\p\Om$.

Thus we can use \cite[Lemma 5.1]{carstea2023inverseproblemgeneralminimal} to say that $\phi$ is the identity on $\p\Om$ to infinite order. This together with \eqref{lambda_times_J} implies that $\lambda$ is equal to one on $\p\Om$ to infinite order.

\section{Second linearization}\label{sec_second_lin}

Using the same notation as in Section \ref{sec_first_lin}, we calculate the second linearization of \eqref{PMC}. Now the second linearization is
\begin{equation}\label{second_lin_F}
	\p_a\big(\p_{p_cp_b}F^a(\nabla u_0)\p_cv_1\p_bv_2 + \p_{p_b}F^a(\nabla u_0)\p_bw\big)=0,
\end{equation}
where $w=\p_{\e_1\e_2}u|_{\e=0}$. Next we calculate $\p_{p_cp_b}F^a(p)$:
\begin{align*}
	\p_{p_cp_b}F^a(p) &= -\frac{p_c}{(1+\abs{p}^2)^{3/2}}\left(\delta_{ab}-\frac{1}{1+\abs{p}^2}p_ap_b\right) \\
	&+ \frac{1}{(1+\abs{p}^2)^{1/2}}\left(\frac{2p_ap_bp_c}{(1+\abs{p}^2)^2} - \frac{\delta_{ac}p_b + \delta_{bc}p_a}{1+\abs{p}^2}\right)\\
	&= -\frac{p_c}{(1+\abs{p}^2)^{3/2}}\left(\delta_{ab}-\frac{1}{1+\abs{p}^2}p_ap_b\right) \\
	&+ \frac{1}{(1+\abs{p}^2)^{3/2}}\left(-p_b\left(\delta_{ac} - \frac{p_ap_c}{1+\abs{p}^2}\right) - p_a\left(\delta_{bc} - \frac{p_bp_c}{1+\abs{p}^2}\right)\right).
\end{align*}
Using \eqref{first_deriv_F} and \eqref{def_g_0} gives
\begin{align*}
	\p_{p_cp_b}F^a(\nabla u_0)=-\frac{1}{1+\abs{\nabla u_0}^2}(\p_au_0g^{bc} + \p_bu_0 g^{ac} + \p_cu_0 g^{ab}).
\end{align*}
Define, again with indices up,
\begin{equation*}
	C^{abc}:=\p_{p_cp_b}F^a(\nabla u_0).
\end{equation*}
Notice that $\p_{p_cp_b}F^a(p)$ is symmetric.

\subsection{Integral identity}

Here we derive an integral identity for the second linearization. Integrating \eqref{second_lin_F} against a solution $v_0$ of the first linearization to obtain
\begin{align*}
	0=&\int_{\Om}\p_a\big(C^{abc}\p_cv_1\p_bv_2 + g_0^{ab}\p_bw\big) v_0\,dx\\
	=& -\int_{\Om} C^{abc}\p_cv_1\p_bv_2\p_av_0\,dx - \int_{\Om}g_0^{ab}\p_bw\p_av_0\,dx \\
	&+ \int_{\p\Om} \big(C^{abc}\p_cv_1\p_bv_2 + g_0^{ab}\p_bw\big) v_0\nu_a\,dS\\
	=& -\int_{\Om} C^{abc}\p_cv_1\p_bv_2\p_a v_0\,dx\\
	&+ \int_{\p\Om} \big(C^{abc}\p_cv_1\p_bv_2 + g_0^{ab}\p_bw\big) v_0\nu_a - wg_0^{ab}\p_a v_0\nu_b\,dS\\
	&= -I + B.
\end{align*}
For the moment let us only focus on $I$:
\begin{align*}
	I&=\int_{\Om} -\frac{1}{1+\abs{\nabla u_0}^2}(\p_au_0g^{bc} + \p_bu_0 g^{ac} + \p_cu_0 g^{ab}) \p_cv_1\p_bv_2\p_a v_0\,dx\\
	&= \int_{\Om} -\frac{1}{1+\abs{\nabla u_0}^2}(\p_au_0\p_a v_0g(\nabla v_1,\nabla v_2) + \p_bu_0\p_bv_2g(\nabla v_0,\nabla v_1) + \p_cu_0\p_c v_1g(\nabla v_0,\nabla v_2))\,dx\\
	&= \int_{\Om} -\frac{\nabla u_0}{1+\abs{\nabla u_0}^2}\cdot(\nabla v_0g(\nabla v_1,\nabla v_2) + \nabla v_1g(\nabla v_0,\nabla v_2) + \nabla v_2g(\nabla v_0,\nabla v_1))\, dx.
\end{align*}
Next consider two sources $H, \tilde H$ as in the previous section. Now by making a change of coordinates with $\phi$
\begin{align*}
	\tilde I &= \int_{\Om} -\frac{1}{1+\abs{\nabla\tilde u_0}^2}(\p_a\tilde u_0\tilde g^{bc} + \p_b\tilde u_0 \tilde g^{ac} + \p_c\tilde u_0 \tilde g^{ab}) \p_c\tilde v_1\p_b\tilde v_2\p_a \tilde v_0\,dx\\
	&= \int_{\Om} -\frac{1}{1+\abs{\nabla \tilde u_0}^2|_{\phi}}((\p_a\tilde u_0)|_{\phi}\tilde g^{bc}|_{\phi} + (\p_b\tilde u_0)|_{\phi} \tilde g^{ac}|_{\phi} + (\p_c\tilde u_0)|_{\phi} \tilde g^{ab}|_{\phi})\\
	&\times [D\phi^{-1}]^{c'}_{c}[D\phi^{-1}]^{b'}_{b}[D\phi^{-1}]^{a'}_{a}\p_{c'}v_1\p_{b'}v_2\p_{a'}v_0J_{\phi}\,dx\\
	&= \int_{\Om} -\frac{\lambda J_{\phi}}{1+\abs{\nabla \tilde u_0}^2|_{\phi}} ((\phi^*\nabla\tilde u_0)^{a'} g^{b'c'} + (\phi^*\nabla\tilde u_0)^{b'} g^{a'c'} + (\phi^*\nabla\tilde u_0)^{c'} g^{a'b'})\\
	&\times \p_{c'}v_1\p_{b'}v_2\p_{a'}v_0\,dx\\
	&= \int_{\Om} -\frac{\lambda J_{\phi} \phi^*\nabla \tilde u_0}{1+\abs{\nabla \tilde u_0}^2|_{\phi}}\cdot(\nabla v_0g(\nabla v_1,\nabla v_2) + \nabla v_1g(\nabla v_0,\nabla v_2) + \nabla v_2g(\nabla v_0,\nabla v_1))\, dx,
\end{align*}
where we used \eqref{first_lin_determination} for $g^{-1}, \tilde g^{-1}$ i.e. $g^{-1}=\lambda^{-1}(D\phi)^{-1}\tilde g^{-1}(D\phi)^{-T}$. Let 
\begin{equation}\label{T_def}
	T:=\left(\frac{\lambda J_{\phi} \phi^*\nabla \tilde u_0}{1+\abs{\nabla \tilde u_0}^2|_{\phi}} - \frac{\nabla u_0}{1+\abs{\nabla u_0}^2}\right)\abs{g}^{-1/2}.
\end{equation}
Then
\begin{align}\label{int_identity}
	I-\tilde I &= \int_{\Om} T\cdot(\nabla v_0g(\nabla v_1,\nabla v_2) + \nabla v_1g(\nabla v_0,\nabla v_2) + \nabla v_2g(\nabla v_0,\nabla v_1))\, dV_g \\\notag
	&= \int_{\Om} T\cdot \bigg(\nabla v_0g(\nabla v_1,\nabla v_2) + \nabla v_1\left(\Delta_g(v_0v_2) - v_2\Delta_gv_0 - v_0\Delta_gv_2\right) \\\notag
	& + \nabla v_2\left(\Delta_g(v_1v_0) - v_0\Delta_gv_1 - v_1\Delta_gv_0\right)\bigg)\, dV_g.
\end{align}

Now if we can show that $T=0$ which further implies, using \eqref{lambda_times_J} and \eqref{norm_u_tilde_u},
\begin{equation}\label{relation_u_tilde_u_3}
	\phi^*\nabla\tilde u_0=\nabla u_0.
\end{equation}

\subsection{Showing that $\phi$ is the identity}\label{sub_sec_phi_identity}

From the first linearization we know that $\lambda D\phi^{-1}\tilde g^{-1}|_{\phi}D\phi^{-T}=g^{-1}$ or equivalently (using the explicit formula \eqref{def_g_0} for $g$ and $\tilde g$)
\begin{align*}
	\lambda D\phi^{-1}D\phi^{-T} - \frac{D\phi^{-1}(\nabla \tilde u_0\otimes\nabla \tilde u_0)|_{\phi}D\phi^{-T}}{1+\abs{\nabla\tilde u_0}^2|_{\phi}} = I - \frac{\nabla u_0\otimes \nabla u_0}{1+\abs{\nabla u_0}^2}.
\end{align*}
Here we used \eqref{norm_u_tilde_u}. Using the above together with \eqref{relation_u_tilde_u_3} and \eqref{norm_u_tilde_u} implies
\begin{align*}
	\lambda D\phi^{-1}D\phi^{-T} = I \quad \text{or}\quad D\phi^{-1}D\phi^{-T}=\lambda^{-1}I.
\end{align*}
That is $\phi$ is a conformal mapping in Euclidean space. 
Combining this with $\phi|_{\p\Om}=\mathrm{Id}$ a Liouville type result (see e.g. \cite[Proposition 3.3.]{Lionheart_1997}) gives that $\phi$ is the identity in $\Om$.

\subsection{Boundary terms}

When deriving the integral identity above we encountered the following integral on the boundary
\begin{equation*}
	B = \int_{\p\Om} \big(C^{abc}\p_cv_1\p_bv_2 + g_0^{ab}\p_bw\big) v_0\nu_a - wg_0^{ab}\p_a v_0\nu_b\,dS.
\end{equation*}
We know the boundary values of $u_0$ and $\p_{\nu}u_0$, hence we know $\nabla u_0$ on the boundary. Thus by the definition of $g$ we know $g$ on the boundary. This further implies that we know $C$ on the boundary. Since we know that the DN map is smooth and we know the DN map related to \eqref{PMC}, we thus know the DN map of the linearized equation. This also uses that $g$ is known on the boundary.

Using these and that $\phi$ is the identity on the boundary to high order (this follows from \cite[Lemma 6.1]{carstea2024calderonproblemriemanniansurfaces}, see Section \ref{sec_boundary})
we have that $\lambda$ is equal to $1$ on the boundary. Putting all of the above together we can conclude that
\begin{equation*}
	B-\tilde B = 0.
\end{equation*}

\section{Complex geometrical optics solutions}\label{sec:CGOs}
In this section, we construct
	CGO solutions for the first linearized equation \eqref{first_lin_divergence} $\nabla\cdot\left(g^{-1}\nabla v\right)=0$. As already noted in \eqref{first_lin_schrödinger} it is sufficient to construct solutions to the Shr\"odinger equation
\begin{equation*}
		(\Delta_g + q) \mathbf{v}=0  \text{ in } \Om
\end{equation*}
and then redefine $v=\gamma^{-1/2}\mathbf{v}$. 
The construction is based on \cite{guillarmou2011identification, guillarmou2011calderon} and the estimates we recall also borrows from \cite{carstea2024calderonproblemriemanniansurfaces,liimatainen2023calder}. We accomplish the construction on Riemannian surfaces $(\Sigma,g)$.  While we do this, we explain the corresponding definitions and concepts in $\C$. When we apply the CGOs, we will pass to global isothermal coordinates. The reader can think that $\Sigma=\Om$ and that the metric is conformally Euclidiean $g=cI_{2\times 2}$. We extend $(\Sigma,g)$ to a slightly larger Riemannian surface $(\widetilde \Sigma,g)$ , and $q$ onto it so that $q\in C_c^\infty(\widetilde \Sigma)$. (We denote the extensions of $g$ and $q$ using the same letters.) 

Let us then recall the standard holomorphic calculus on Riemannian surfaces. The complexified cotangent bundle $\C T^*\widetilde M$ has the splitting
\[
 \C T^*\widetilde M= T^*_{1,0}\widetilde M \oplus T^*_{0,1}\widetilde M
\]
determined by the eigenspaces of the Hodge star operator $\star$.  In local holomorphic coordinate $z$ the space $T^*_{1,0}\widetilde \Sigma$ is spanned by $dz$ and $T^*_{0,1}\widetilde \Sigma$ is spanned by $d\ol z$.  The invariant definitions of $\op$ and $\p$ operators are $\pi_{0,1}d$  and  $\pi_{1,0}d$ respectively.
	In holomorphic coordinates $z=(x,y)$ these operators are 
	\[
	 \op=\frac 12 (\p_x+i\p_y) \ \text{ and } \ \p=\frac 12  (\p_x-i\p_y). 
	\]
	By \cite[Proposition 2.1]{guillarmou2011identification} there is a right inverse $\op^{-1}$ for $\op$ in the sense that
	\[
	\op\s \s  \op^{-1}\omega=\omega \text{ for all } \omega\in C_0^\infty(\widetilde \Sigma,T_{1,0}^*\widetilde \Sigma)
	\]
	such that $\op^{-1}$ is bounded from $L^p(T_{1,0}^*\widetilde \Sigma)$ to $W^{1,p}(\widetilde \Sigma)$ for any $p\in (1,\infty)$. We have analogous properties for 
	\[
	\op^*=-\mathsf{i}\star \p: W^{1,p}(T_{0,1}^*\widetilde \Sigma)\to L^p(\widetilde \Sigma),
	\]
	which is the Hermitean adjoint of $\op$. In holomorphic coordinates  the operator  $\op^*$  is just  $\p$. If $\Sigma$ admits global holomorphic coordinates, then in those coordinates  $\op^{-1}$ and $\op^{*-1}$ equal
	\[
	 (\overline{\partial}^{-1} \mathbf{u})(z)=\int_{\C} \frac{\mathbf{u}(w)}{z-w} \mathrm{~d} w \ \text{ and } \  ({\partial}^{-1} \mathbf{u})(z)=\int_{\C}  \frac{\mathbf{u}(w)}{\overline z-\overline w} \mathrm{~d} w
	\]
	with the understanding that the function $\mathbf{u}\in L^p(\Sigma)$ in the coordinates are continued onto $\C$ in $L^p$ (and similarly if $\mathbf{u}$ has higher regularity). We also define 
	\[
	\op_\psi^{-1}:=\mathcal{R}\op^{-1}e^{-2i\psi/h}\mathcal{E} \quad \text{and}\quad  \op_\psi^{*-1}:=\mathcal{R}\op^{*-1}e^{2i\psi/h}\mathcal{E},
	\]
	where $\mathcal{E} : W^{l,p}(\Sigma) \to W_c^{l,p}(\widetilde \Sigma)$ is an extension operator from functions on $\Sigma$ to functions on the extended surface $\widetilde \Sigma$ and $\mathcal R$ is the restriction operator back to $\Sigma$. By \cite[Lemma 2.2 and Lemma 2.3]{guillarmou2011identification} we have for $p>2$ and $2\leq q\leq p$ the following estimate 
\begin{align}\label{eq:sobo_decay}
\begin{split}
  \norm{\overline \p_\psi^{-1}f}_{L^q(M)}&\leq C h^{1/q}\norm{f}_{W^{1,p}(M, T_{0,1}^*M)} 
 \end{split}
\end{align}
Moreover, there is $\eps>0$ such that 
\begin{align}\label{eq:sobo_decayL2}
\begin{split}
 \norm{\overline \p_\psi^{-1}f}_{L^2(M)}&\leq C h^{1/2+\eps}\norm{f}_{W^{1,p}(M, T_{0,1}^*M)} 
 \end{split}
\end{align}
We have similar estimates for $\op_\psi^{*-1}f$.

The CGOs on a Riemannian surface $(\Sigma,g)$ for the equation $(\Delta_g + q)\mathbf{v}=0$ have the form
\begin{equation}\label{eq:CGO_def}
			\mathbf{v}=e^{\Phi/h}(a+r_h).
\end{equation}
Here $\Phi$ 
 is holomorphic function, which is Morse (i.e. Hessian at possible critical points of $\Phi$ is invertible). Here also $a$ is a holomorphic function and $r=r_h$ is a correction term. The precise form of $r_h$ is 
\begin{equation}\label{eq:rh_form}
 r_h=-  \overline{\p}_\psi^{-1}\sum_{j=0}^\infty T_h^j\op_\psi^{*-1}(qa),
\end{equation}
where the operators $\overline{\p}_\psi^{-1}$ and $\op_\psi^{*-1}$ were defined above and $T_h$ is defined as
\begin{eqnarray}\label{def: Th}
 T_h:=-\op_\psi^{*-1}q\overline\p_\psi^{-1}.
\end{eqnarray}
We also denote
\begin{equation}\label{eq:sh_formula}
s_h=\sum_{j=0}^\infty T_h^j\op_\psi^{*-1}(qa)
\end{equation}
so that $r_h=-\overline{\p}_\psi^{-1}s_h$.

By taking complex conjugate and changing the sign of the phase function, we also have CGO solutions with antiholomorphic phase of the form
\begin{equation}\label{eq:CGO_def_antiholom}
			\mathbf{v}=e^{-\overline \Phi/h}(b+\tilde r_h),
\end{equation}
where $\overline \Phi$ is antiholomorphic Morse function, $b$ antiholomorphic and $\tilde r_h$ given by:
\begin{equation}\label{eq:tilde_rh_formula}
 \tilde r_h=-\p_\psi^{-1}\sum_{j=0}^\infty \tilde T_h^j(\p_{\psi}^{*-1} (q a))=-\p_\psi^{-1}\tilde s_h.
\end{equation}

Let us then recall estimates from \cite{guillarmou2011identification, guillarmou2011calderon, carstea2024calderonproblemriemanniansurfaces,liimatainen2023calder}. If $\Phi$ has a critical point, then  
\begin{equation}\label{eq:CZ_rh_norms}
||r_h||_{L^p}, ||s_h||_{L^p},   ||\nabla r_h||_{L^p}=O(h^{1/p+\epsilon}),
\end{equation}
for any $p\in[2,\infty)$ and $0<\eps\leq \epsilon_p$ depending on $p$. Moreover, we have
\begin{eqnarray}
\label{Th norm estimate}
 \norm{T_h}_{L^r\to L^r}=O(h^{1/r}) \text{ and } \norm{T_h}_{L^2\to L^2}=O(h^{1/2-\eps}), 
\end{eqnarray}
for any $0<\eps<1/2$ and $r>2$. 
See  \cite[Section 4.1]{carstea2024calderonproblemriemanniansurfaces} for the above. If
 $f\in C^\infty(\Sigma)$ vanishes to order $1$ on $\p \Sigma$ and $\hat \psi $ is the imaginary part of holomorphic Morse function on $\Sigma$, then 
\begin{equation}\label{eq_integrate_by_parts_doo}
  \int_\Sigma e^{i\hat \psi/h}f r_h=-\int_\Sigma e^{i\hat \psi/h}f \overline{\p}_\psi^{-1}s_h=\int_\Sigma  \op^{-1}(e^{i\hat \psi/h}f)e^{-2i\psi/h} s_h=O(h^{1+\eps})
 \end{equation} 
 by \eqref{eq:sobo_decayL2}. 
 Here $\op^{-1}(e^{i\hat \psi/h}f)$ means the Cauchy-Riemann operator $\op^{-1}$ applied to the zero extension of $f$ onto $\widetilde \Sigma$, which is $C^1(\widetilde\Sigma)\subset W^{1,p}(\widetilde\Sigma)$ for all $p>1$. The same estimate holds when $r_h$ is replaced by $\tilde r_h$.

If $\Phi$ has no critical points, we have the better estimates
\begin{eqnarray}
\label{eq: no critical point estimate}
\|r_h\|_p +\|s_h\|_p+ \|d r_h\|_p \leq Ch
\end{eqnarray}
for all $p\in (1,\infty)$. Also, in this case we have the expansion
\begin{eqnarray}\label{eq: no critical point expansion II} \op_\psi^{-1} f =e^{-2i\psi/h} \frac{ih}{2} \frac{f}{\bar\partial\psi} + \frac{ih}{2}\op^{-1}\left( e^{-2i\psi/h}\bar\partial\left(\frac{f}{\bar\partial\psi}\right)\right),
\end{eqnarray} 
and the $L^2$ norm second order derivatives of $r_h$ multiplied by any $H\in C_0^{\infty}(\widetilde \Sigma)$ satisfy
\begin{equation}\label{eq_cald_zygm}
||H\nabla^2  r_{h}||_{L^2} =O(1)
\end{equation}
see \cite[Section 4.1]{carstea2024calderonproblemriemanniansurfaces} and \cite[Eq. (3.14)]{liimatainen2023calder}. 

We have similar estimates  for $\tilde r_h$, $\tilde s_h$ and $\tilde T_h$, and expansion for $\op_\psi^{*-1}$, to the ones above.

\section{Proving that $T=0$}\label{sec_T=0}

Here we assume that $T$ vanishes to high order on the boundary $\p\Om$. The integral identity \eqref{int_identity} holds for arbitrary solutions of the first linearization \eqref{first_lin_divergence}. We will construct these to be of the form $v=\gamma^{-1/2}\hat{v}$, $\gamma=\abs{g}^{-1/2}$, where $\hat{v}$ is a solution of \eqref{first_lin_schrödinger}. Thus we can modify by noting that for these solutions 
\begin{align*}
	\Delta_g v_k &= \Delta_g(\gamma^{-1/2}\hat{v}_k) = \Delta_g(\gamma^{-1/2})\hat{v}_k + 2g\left(\nabla(\gamma^{-1/2}),\nabla \hat{v}_k\right) + \gamma^{-1/2}\Delta_g\hat{v}_k\\
	&= \Delta_g(\gamma^{-1/2})\hat{v}_k + 2g\left(\nabla(\gamma^{-1/2}),\nabla \hat{v}_k\right) - \gamma^{-1/2}q\hat{v}_k\\
	&= 2g\left(\nabla(\gamma^{-1/2}),\nabla \hat{v}_k\right) + \left(\Delta_g(\gamma^{-1/2}) - \gamma^{-1/2}q\right)\hat{v}_k,
\end{align*}
where $q=\abs{g}^{1/4}\Delta_g(\abs{g}^{-1/4})$:
\begin{align*}
	I-\tilde I &= \sum_{k=1}^{2}\int_{\Om}T^k\Big( \p_kv_0g(\nabla v_1, \nabla v_2) \\
	&+ \p_kv_1(\Delta_g(v_0v_2) - \gamma^{-1/2}\hat{v}_2\Delta_g(\gamma^{-1/2}\hat{v}_0) - \gamma^{-1/2}\hat{v}_0\Delta_g(\gamma^{-1/2}\hat{v}_2))\\
	&+ \p_kv_2(\Delta_g(v_0v_1) - \gamma^{-1/2}\hat{v}_0\Delta_g(\gamma^{-1/2}\hat{v}_1) - \gamma^{-1/2}\hat{v}_1\Delta_g(\gamma^{-1/2}\hat{v}_0))\Big)\,dV_g\\
	&= \sum_{k=1}^{2} \int_{\Om} T^k\p_kv_0g(\nabla v_1, \nabla v_2) \\
	&+ T^k\p_kv_1\left(\Delta_g(v_0v_2) + 2\left(g\left(\nabla(\gamma^{-1/2}),\nabla \hat{v}_0\right)\hat{v}_2 + g\left(\nabla(\gamma^{-1/2}),\nabla \hat{v}_2\right)\hat{v}_0\right) + 2Q\hat{v}_0\hat{v}_2\right)\\
	&+ T^k\p_kv_2\left(\Delta_g(v_0v_1) + 2\left(g\left(\nabla(\gamma^{-1/2}),\nabla \hat{v}_0\right)\hat{v}_1 + g\left(\nabla(\gamma^{-1/2}),\nabla \hat{v}_1\right)\hat{v}_0\right) + 2Q\hat{v}_0\hat{v}_1\right)\,dV_g\\
	&:= I_1 + I_2 + I_3.
\end{align*}
Here $Q:=\Delta_g(\gamma^{-1/2}) - \gamma^{-1/2}q$.
We choose the functions $\hat{v}_k$ as follows: Let $\Psi, \Phi, a_k$ be holomorphic Morse functions, $k=1,2$, and $a_0$ antiholomorphic such that $\pm\Psi + \Phi$ has no critical points and $\Phi$ has critical points, see Section \ref{sec:CGOs}. Then
\begin{align*}
	\hat{v}_0 &= e^{-2\ol{\Phi}/h}(a_0 + r_0)\\
	\hat{v}_1 &= e^{(\Psi+\Phi)/h}(a_1 + r_1)\\
	\hat{v}_2 &= e^{(-\Psi+\Phi)/h}(a_2 + r_2)
\end{align*}
and
\begin{equation*}
	v_k = \gamma^{-1/2}\hat{v}_k\quad\text{ for } k=0,1,2.
\end{equation*}
At this point it is beneficial to use holomorphic coordinates around an interior point $p\in\Om$. In these coordinates $z=x+iy$, we have $g(z)=c(z)dzd\oz$ where $c$ is a smooth function, $dz=dx+idy$ and $d\oz=dx-idy$. Recall also the operators
\begin{equation}\label{doo_doobar}
	\p = \frac{1}{2}(\p_x - i\p_y)\quad \text{and}\quad \ol{\p} = \frac{1}{2}(\p_x + i\p_y),
\end{equation}
thus $\nabla_{(x,y)}=(\p+\ol{\p},i(\p-\ol{\p}))$.
In these coordinates
\begin{equation}\label{holomorp_coord}
	g(\nabla u,\nabla v)=2c^{-1}(\p u\dbar v + \dbar u\p v),\quad\Delta_gu = -2ic^{-1}\ol{\p}\p u\quad\text{and}\quad dV_g=c(z)\frac{i}{2}dzd\oz.
\end{equation}

Since $v_1$ and $v_2$ have a holomorphic phase, they act similarly in the asymptotic analysis, and thus it is enough to do analysis for $I_2$ as the analysis for $I_3$ is the same and in the end we get the same terms from stationary phase (see Lemmas \ref{lemma_A}, \ref{lemma_B}, \ref{lemma_C} below). Using the above, we have for $I_2$ after integrating by parts
\begin{align*}
	I_2 &= \sum_{k=1}^2 \int_{\Om}T^k\p_k v_1 \ol{\p}\p(v_0v_2)\,dzd\oz\\
	&+ i\int_{\Om} cT^k \p_k v_1 \left(g\left(\nabla(\gamma^{-1/2}),\nabla \hat{v}_0\right)\hat{v}_2 + g\left(\nabla(\gamma^{-1/2}),\nabla \hat{v}_2\right)\hat{v}_0\right)\,dzd\oz\\
	&+ i\int_{\Om} cT^k\p_kv_1Q\hat{v}_0\hat{v}_2\,dzd\oz\\
	&= I_{2,1} + I_{2,2} + I_{2,3}
\end{align*}
and similarly for $I_3$. Here we make a slight abuse of notation by writing again $T^k$ for its representation in holomorphic coordinates.

In the analysis of the above integrals we will use stationary phase arguments for terms including $T$. For this we need that $T$ vanishes to high order on the boundary $\p\Om$. Recall that 
\begin{equation*}
	T:=\left(\frac{\lambda J_{\phi} \phi^*\nabla \tilde u_0}{1+\abs{\nabla \tilde u_0}^2|_{\phi}} - \frac{\nabla u_0}{1+\abs{\nabla u_0}^2}\right)\abs{g}^{-1/2}.
\end{equation*}
As explained in Section \ref{sec_boundary} we have that $\Phi$ is the identity, $\lambda=1$ and $u_0=\tilde u_0$ to high order on $\p\Om$. Thus $T$ vanishes to high order on $\p\Om$. We are now ready to prove the following:

\begin{lem}\label{lemma_A}
	Define $A_5=2T^kc^{-1}\dbar(\gamma^{-1/2}) + \dbar T^k\gamma^{-1/2}$. Then
	$$\lim_{h\to 0} hI_{2,1} = \sum_{k=1}^2Ci^{k-1}(A_5\gamma^{-1})|_{z=z_0}$$
	for some constant $C>0$.
\end{lem}

\begin{proof}
	Integrating by parts in $I_{2,1}$ we have
	\begin{align*}
		I_{2,1} &= \sum_{k=1}^2 \int_{\Om}\ol{\p}\p(T^k\p_kv_1)v_0v_2\,dzd\oz\\
		&= \sum_{k=1}^2 \int_{\Om}\ol{\p}\p(T^k)\p_kv_1v_0v_2 + T^k\ol{\p}\p(\p_kv_1)v_0v_2 + \ol{\p}T^k\p\p_k(v_1)v_0v_2 + \p T^k\ol{\p}\p_kv_1v_0v_2\,dzd\oz\\
		&=  \sum_{k=1}^2 \int_{\Om}\ol{\p}\p(T^k)\p_kv_1v_0v_2 + T^k\p_k(2g\left(\nabla(\gamma^{-1/2}),\nabla \hat{v}_1\right) + Q\hat{v}_1)v_0v_2\\
		& + \ol{\p}T^k\p\p_k(v_1)v_0v_2 + \p T^k\ol{\p}\p_kv_1v_0v_2\,dzd\oz.
	\end{align*}
	Using \eqref{doo_doobar} and \eqref{holomorp_coord} we can group the integral $I_{2,1}$ as follows:
	\begin{align*}
		I_{2,1}&= \sum_{k=1}^2 \int_{\Om}\Big(\dbar\p(T^k)\p_k(\gamma^{-1/2}) + T^k\p_kQ + \dbar T^k\p\p_k(\gamma^{-1/2}) + \p T^k \dbar\p_k(\gamma^{-1/2})\Big)\hat{v}_1v_0v_2\\
		&+ \Big(i^{k-1}\dbar\p(T^k)\gamma^{-1/2} + 2T^k\p_k(c^{-1})\dbar(\gamma^{-1/2}) + 2T^kc^{-1}\p_k\dbar(\gamma^{-1/2}) + i^{k-1}T^kQ \\
		&+ \dbar T^k\p_k(\gamma^{-1/2}) + i^{k-1}\dbar T^k\p(\gamma^{-1/2}) + i^{k-1}\p T^k\dbar(\gamma^{-1/2})\Big)\p \hat{v}_1v_0v_2\\
		&+\Big((-i)^{k-1}\dbar\p(T^k)\gamma^{-1/2} + 2T^k\p_k(c^{-1})\p(\gamma^{-1/2}) + 2T^kc^{-1}\p_k\p(\gamma^{-1/2}) + (-i)^{k-1}T^kQ \\
		&+ (-i)^{k-1}\dbar T^k\p(\gamma^{-1/2}) + \p T^k\p_k(\gamma^{-1/2}) + (-i)^{k-1}\p T^k \dbar(\gamma^{-1/2})\Big)\dbar \hat{v}_1v_0v_2\\
		&+\Big(2T^kc^{-1}\p(\gamma^{-1/2}) + \p T^k\gamma^{-1/2}\Big)\dbar\p_k \hat{v}_1v_0v_2 \\
		&+ \Big(2T^kc^{-1}\dbar(\gamma^{-1/2}) + \dbar T^k\gamma^{-1/2}\Big)\p\p_k \hat{v}_1v_0v_2\,dzd\oz\\
		&= \sum_{k=1}^2\int_{\Om} A_1\hat{v}_1v_0v_2 + A_2\p \hat{v}_1v_0v_2 + A_3\dbar \hat{v}_1v_0v_2 + A_4\dbar\p_k \hat{v}_1v_0v_2 + A_5\p\p_k \hat{v}_1v_0v_2\,dzd\oz.
	\end{align*}
    All of the terms $A_j$ depend on the summation index $k$.
	
	In the following we do not write the summation over $k$. Next we further analyze $I_{2,1}$ term by term, starting from the term including $A_1$:
	\begin{align*}
		\int_{\Om}A_1\hat{v}_1v_0v_2\,dzd\oz &= \int_{\Om}\gamma^{-1}A_1e^{(2\Phi - 2\Phi)/h}(a_0 + r_0)(a_1 + r_1)(a_2 + r_2)\,dzd\oz = O(1)
	\end{align*}
	as $h\to0$, since the functions $\gamma, A_1, a_j$ are bounded in $\Om$ and by the estimates for $r_j$ (see \eqref{eq:CZ_rh_norms} and \eqref{eq: no critical point estimate}). Next the term with $A_2$:
	\begin{align*}
		&\int_{\Om}A_2\p \hat{v}_1v_0v_2\\
		&= \frac{1}{h}\int_{\Om}\gamma^{-1}A_2\p(\Psi + \Phi)e^{(2\Phi - 2\ol{\Phi})/h}(a_0 + r_0)(a_1 + r_1)(a_2 + r_2)\,dzd\oz \\
		&+ \int_{\Om}\gamma^{-1}A_2e^{(\Psi + \Phi)/h}(\p a_1 + \p r_1)\hat{v}_0\hat{v}_2\,dzd\oz\\
		&=\frac{1}{h}\int_{\Om}\gamma^{-1}A_2 \p(\Psi + \Phi)e^{(2\Phi - 2\ol{\Phi})/h}a_0a_1a_2\,dzd\oz + O(h^{-1/2+\e}) + O(1),
	\end{align*}
	where the term $O(h^{-1/2+\e})$ comes from the fact that all the terms include at least one $r_j$ and the term $O(1)$ similarly as for the term with $A_1$. The last term will be $O(1)$ by stationary phase and hence
	\begin{equation}\label{A_2}
		\int_{\Om}A_2\p \hat{v}_1v_0v_2 = O(h^{-1/2+\e}).
	\end{equation}
	The term with $A_3$ is a little bit simpler since $\dbar \Phi= \dbar\Psi = \dbar a_1=0$ by the holomorphicity of these functions. Thus
	\begin{align*}
		\int_{\Om} A_3\dbar \hat{v}_1v_0v_2\,dzd\oz = \int_{\Om}\gamma^{-1}A_3e^{(\Psi + \Phi)/h}\dbar r_1\hat{v}_0\hat{v}_2\,dzd\oz = O(h)
	\end{align*}
	because of the term $\dbar r_1$, see \eqref{eq: no critical point estimate}. Continuing with $A_4$:
	\begin{align*}
		\int_{\Om} A_4\dbar \p_k \hat{v}_1 v_0v_2\,dzd\oz &= \int_{\Om}\gamma^{-1}A_4\left(i^{k-1}\dbar\p \hat{v}_1\hat{v}_0\hat{v}_2 + (-i)^{k-1}\dbar^2 \hat{v}_1\hat{v}_0\hat{v}_2\right)\,dzd\oz\\
		&= \int_{\Om}\gamma^{-1}A_4\frac{i^{k-2}}{2}cq\hat{v}_1\hat{v}_0\hat{v}_2\,dzd\oz \\
		&+ \int_{\Om}\gamma^{-1}A_4(-i)^{k-1}\dbar^2\left(e^{(\Psi+\Phi)/h}(a_1+r_1)\right)\hat{v}_0\hat{v}_2\,dzd\oz\\
		&= O(1) + \int_{\Om}\gamma^{-1}A_4(-i)^{k-1}\dbar^2\left(e^{(\Psi+\Phi)/h}(a_1+r_1)\right)\hat{v}_0\hat{v}_2\,dzd\oz
	\end{align*}
	and this can be seen by the same arguments as for the term with $A_1$. For the rest, we notice that each time the operator $\dbar$ hits $e^{(\Psi+\Phi)/h}$ we get zero since $\dbar \Phi= \dbar\Psi=0$. Also $\dbar a_1=0$, hence
	\begin{align*}
		&\int_{\Om}\gamma^{-1}A_4(-i)^{k-1}\dbar^2\left(e^{(\Psi+\Phi)/h}(a_1+r_1)\right)\hat{v}_0\hat{v}_2\,dzd\oz\\
		&= \int_{\Om} \gamma^{-1}A_4(-i)^{k-1}e^{(\Psi+\Phi)/h}\dbar^2r_1\hat{v}_0\hat{v}_2\\
		&= O(1)
	\end{align*}
	by the same arguments as before and using the Calderón-Zygmund type estimate in \eqref{eq_cald_zygm} for $\dbar^2 r_1$.
	
	The term with $A_5$ is the most complicated of these. First of all
	\begin{align*}
		&\int_{\Om}A_5\p\p_k \hat{v}_1v_0v_2\,dzd\oz\\
		&= \int_{\Om}A_5i^{k-1}\p^2\hat{v}_1v_0v_2\,dzd\oz + \int_{\Om}A_5(-i)^{k-1}\p\dbar \hat{v}_1 v_0v_2\,dzd\oz\\
		&= \int_{\Om}A_5\gamma^{-1}i^{k-1}\p\left(\frac{1}{h}\p(\Psi + \Phi)e^{(\Psi + \Phi)/h}(a_1+r_1) + e^{(\Psi + \Phi)/h}(\p a_1 + \p r_1)\right) \hat{v}_0\hat{v}_2\,dzd\oz\\
		&+ O(1)
	\end{align*}
	by the same arguments as for the integral with $A_4$. Expanding further
	\begin{align*}
		&\int_{\Om}A_5\p\p_k \hat{v}_1v_0v_2\,dzd\oz\\
		&= \int_{\Om}A_5\gamma^{-1}i^{k-1}\frac{1}{h^2}(\p(\Psi + \Phi))^2e^{(\Psi + \Phi)/h}(a_1+r_1)\hat{v}_0\hat{v}_2\,dzd\oz\\
		&+ \int_{\Om}A_5\gamma^{-1}i^{k-1}\frac{1}{h}\p^2(\Psi + \Phi)e^{(\Psi + \Phi)/h}(a_1+r_1)\hat{v}_0\hat{v}_2\,dzd\oz\\
		&+ 2\int_{\Om}A_5\gamma^{-1}i^{k-1}\frac{1}{h}\p(\Psi + \Phi)e^{(\Psi + \Phi)/h}(\p a_1+\p r_1)\hat{v}_0\hat{v}_2\,dzd\oz\\
		&+ \int_{\Om}A_5\gamma^{-1}i^{k-1}e^{(\Psi + \Phi)/h}(\p^2a_1+\p^2r_1)\hat{v}_0\hat{v}_2\,dzd\oz + O(1).
	\end{align*}
	Focusing first on the terms with $\frac{1}{h}$, we get that they are $O(h^{-1/2+\e})$ by stationary phase (for terms with no remainders) and that the remainders $r_k$ satisfy the estimates \eqref{eq:CZ_rh_norms}, \eqref{eq: no critical point estimate}. For the ones with $\frac{1}{h^2}$ we have three kinds of terms: one without any $r_k$, ones with just one $r_k$ and ones with two or more $r_k$:
	\begin{align*}
		&\int_{\Om}A_5\gamma^{-1}i^{k-1}\frac{1}{h^2}(\p(\Psi + \Phi))^2e^{(\Psi + \Phi)/h}(a_1+r_1)\hat{v}_0\hat{v}_2\,dzd\oz\\
		&= \frac{1}{h^2}\int_{\Om}A_5\gamma^{-1}i^{k-1}(\p(\Psi + \Phi))^2e^{(2\Phi - 2\ol{\Phi})/h}a_0a_1a_2\,dzd\oz\\
		&+ \frac{1}{h^2}\int_{\Om}A_5\gamma^{-1}i^{k-1}(\p(\Psi + \Phi))^2e^{(2\Phi - 2\ol{\Phi})/h}\Big(a_1a_0r_2 + a_1r_0a_2 + a_0r_1a_2 \\
		&+ a_1r_0r_2 + a_0r_1r_2 + r_0r_1a_2 + r_0r_1r_2\Big)\,dzd\oz.
	\end{align*}
	Because $\Psi$ and $\Phi$ were chosen so that $\p(\Psi + \Phi)(z_0) \neq 0$ we can use stationary phase to conclude that the first term is
	\begin{equation*}
		\frac{1}{h^2}\int_{\Om}A_5\gamma^{-1}i^{k-1}(\p(\Psi + \Phi))^2e^{(2\Phi - 2\ol{\Phi})/h}a_0a_1a_2\,dzd\oz = C\frac{1}{h}i^{k-1}(A_5\gamma^{-1})|_{z=z_0} + O(1)
	\end{equation*}
	for some constant $C>0$. For the terms with at least two $r_k$, we can use the estimates \eqref{eq:CZ_rh_norms}, \eqref{eq: no critical point estimate} for these to conclude that they are $O(h^{-1/2+\e})$ as $h\to 0$. We are then left with the terms that have only one $r_k$. Let us consider the case with $r_0$. The other two are exactly the same and the estimates for $r_1, r_2$ are better than for $r_0$. Now by the definition of $r_0$ we have
	\begin{align}\label{A_5_last}
		&\frac{1}{h^2}\int_{\Om}A_5\gamma^{-1}i^{k-1}(\p(\Psi + \Phi))^2e^{(2\Phi - 2\ol{\Phi})/h}a_1a_2r_0\,dzd\oz\\\notag
		&=-\frac{1}{h^2}\int_{\Om}A_5\gamma^{-1}i^{k-1}(\p(\Psi + \Phi))^2e^{(2\Phi - 2\ol{\Phi})/h}a_1a_2\p^{-1}_{\Imm(-2\ol{\Phi})}s_0\,dzd\oz\\\notag
		%
        &= O(h^{-1+\e}).
	\end{align}
	Here we used that one can integrate by parts with the operator $\p^{-1}_{\psi}$, see \eqref{eq_integrate_by_parts_doo}.
    All in all we have
	\begin{align}\label{A_5}
		\int_{\Om}A_5\p\p_k \hat{v}_1v_0v_2\,dzd\oz = C\frac{1}{h}i^{k-1}(A_5\gamma^{-1})|_{z=z_0} + O(h^{-1+\e})
	\end{align}
	which concludes the proof.
\end{proof}

Now we move to analyze the integral $I_{2,2}$ (i.e. to \eqref{I_22_B}).

\begin{lem}\label{lemma_B}
	Define $B_7=i^{k-1}T^k\dbar(\gamma^{-1/2})$. Then
	\begin{equation*}
		\lim_{h\to 0} hI_{2,2} = \sum_{k=1}^2iC(B_7\gamma^{-1/2})|_{z=z_0}
	\end{equation*}
	for some constant $C>0$.
\end{lem}

\begin{proof}
	For $I_{2,2}$ we have, using the same calculations as for $I_{2,1}$,
	\begin{align}\label{I_22_B}
		-\frac{i}{2}I_{2,2} &= \sum_{k=1}^2\int_{\Om} i^{k-1}T^k\p(\gamma^{-1/2})\p v_1\dbar \hat{v}_0\hat{v}_2 + (-i)^{k-1}T^k\p(\gamma^{-1/2})\dbar v_1\dbar \hat{v}_0\hat{v}_2\\\notag
		&+ i^{k-1}T^k\dbar(\gamma^{-1/2})\p v_1\p \hat{v}_0 \hat{v}_2 + (-i)^{k-1} T^k\dbar(\gamma^{-1/2})\dbar v_1\p \hat{v}_0 \hat{v}_2\\\notag
		&+ i^{k-1}T^k\p(\gamma^{-1/2})\p v_1 \dbar \hat{v}_2 \hat{v}_0 + (-i)^{k-1}T^k\p(\gamma^{-1/2})\dbar v_1\dbar \hat{v}_2 \hat{v}_0\\\notag
		&+ i^{k-1}T^k\dbar(\gamma^{-1/2})\p v_1\p \hat{v}_2\hat{v}_0 + (-i)^{k-1}T^k\dbar(\gamma^{-1/2})\dbar v_1 \p \hat{v}_2\hat{v}_0\,dzd\oz\\\notag
		&= \sum_{k=1}^2\int_{\Om} B_1\p v_1 \dbar \hat{v}_0\hat{v}_2 + B_2\dbar v_1\dbar \hat{v}_0\hat{v}_2 + B_3\p v_1\p \hat{v}_0\hat{v}_2 + B_4\dbar v_1\p \hat{v}_0\hat{v}_2\\\notag
		&+ B_5\p v_1 \dbar \hat{v}_2\hat{v}_0 + B_6\dbar v_1\dbar \hat{v}_2\hat{v}_0 + B_7\p v_1\p \hat{v}_2\hat{v}_0 + B_8\dbar v_1\p \hat{v}_2\hat{v}_0\,dzd\oz.
	\end{align}
	Again the terms $B_j$ depend on $k$. The first term is
	\begin{align*}
		&\int_{\Om} B_1\p v_1\dbar \hat{v}_0 \hat{v}_2\,dzd\oz\\
		&= \int_{\Om}B_1\p(\gamma^{-1/2})\hat{v}_1\dbar \hat{v}_0\hat{v}_2\,dzd\oz + \int_{\Om} B_1\gamma^{-1/2}\p \hat{v}_1\dbar \hat{v}_0 \hat{v}_2\,dzd\oz\\
		&= \int_{\Om}B_1\p(\gamma^{-1/2})e^{(2\Phi - 2\ol{\Phi})/h}(a_1+r_1)\Big(\frac{-2}{h}\dbar\ol{\Phi}(a_0+r_0) + \dbar a_0 + \dbar r_0\Big)(a_2 + r_2)\,dzd\oz\\
		&+ \int_{\Om}B_1\gamma^{-1/2}e^{(2\Phi - 2\ol{\Phi})/h}\Big(\frac{1}{h}\p(\Psi + \Phi)(a_1 + r_1) + \p a_1 + \p r_1\Big)\\
		&\times \Big(\frac{-2}{h}\dbar\ol{\Phi}(a_0+r_0) + \dbar a_0 + \dbar r_0\Big)(a_2 + r_2)\,dzd\oz\\
		&= O(h^{-1/2+\e}) + O(1)
	\end{align*}
	by stationary phase and the estimates for $r_k$ (see \eqref{eq:CZ_rh_norms}, \eqref{eq: no critical point estimate}). For the second term in \eqref{I_22_B} we have
	\begin{align*}
		&\int_{\Om}B_2\dbar v_1\dbar \hat{v}_0\hat{v}_2 \,dzd\oz\\
		&= \int_{\Om}B_2\dbar(\gamma^{-1/2})\hat{v}_1\dbar \hat{v}_0\hat{v}_2\,dzd\oz + \int_{\Om} B_2\gamma^{-1/2}\dbar \hat{v}_1\dbar \hat{v}_0 \hat{v}_2\,dzd\oz\\
		&= O(h^{-1/2}+\e)\\
		&+ \int_{\Om} B_2\gamma^{-1/2}e^{(2\Phi - 2\ol{\Phi})/h}\dbar r_1\Big(\frac{-2}{h}\dbar\ol{\Phi}(a_0+r_0) + \dbar a_0 + \dbar r_0\Big)(a_2 + r_2)\,dzd\oz\\
		&= O(h^{-1/2+\e}) + O(1)
	\end{align*}
	where the first term is as above with $B_1$ and the second is handled with the estimate \eqref{eq: no critical point estimate} for $\dbar r_1$. Next the term with $B_3$:
	\begin{align*}
		&\int_{\Om}B_3 \p v_1 \p \hat{v}_0\hat{v}_2\,dzd\oz\\
		&= \int_{\Om}B_3\p(\gamma^{-1/2})\hat{v}_1\p \hat{v}_0\hat{v}_2\,dzd\oz + \int_{\Om} B_3\gamma^{-1/2}\p \hat{v}_1\p \hat{v}_0 \hat{v}_2\,dzd\oz\\
		&= \int_{\Om}B_3\p(\gamma^{-1/2})e^{(2\Phi - 2\ol{\Phi})/h}\p r_0(a_1+r_1)(a_2+r_2)\,dzd\oz\\
		&+ \int_{\Om} B_3\gamma^{-1/2}e^{(2\Phi - 2\ol{\Phi})/h}\Big(\frac{1}{h}\p(\Psi + \Phi)(a_1 + r_1) + \p a_1 + \p r_1\Big)\p r_0(a_2 + r_2)\,dzd\oz\\
		&= O(h^{1/2+\e}) + O(h^{-1/2+\e})
	\end{align*}
	since all the terms have $\p r_0$ in them (see \eqref{eq:CZ_rh_norms}).
	For the next one we have
	\begin{align*}
		&\int_{\Om}B_4 \dbar v_1 \p \hat{v}_0\hat{v}_2\,dzd\oz\\
		&= \int_{\Om}B_4\dbar(\gamma^{-1/2})\hat{v}_1\p \hat{v}_0\hat{v}_2\,dzd\oz + \int_{\Om} B_4\gamma^{-1/2}\dbar \hat{v}_1\p \hat{v}_0 \hat{v}_2\,dzd\oz\\
		&= O(h^{1/2+\e}) + \int_{\Om} B_4\gamma^{-1/2}e^{(2\Phi - 2\ol{\Phi})/h}\dbar r_1\p r_0(a_2+r_2)\,dzd\oz\\
		&= O(h^{1/2+\e}) + O(h^{3/2+\e})
	\end{align*}
	similarly as above.
	
	For the terms with $B_5$ to $B_8$ we only consider the term with $B_7$ in detail. For $B_5, B_6$ they both have the term $\dbar \hat{v}_2 = e^{(-\Psi + \Phi)/h}\dbar r_2$ and thus are $O(h^{\alpha})$ with $\alpha>-1/2$ (see \eqref{eq: no critical point estimate}). For $B_8$ we have terms like $\hat{v}_1\p \hat{v}_2 \hat{v}_0$ and $\dbar \hat{v}_1 \p \hat{v}_2 \hat{v}_0$. The first is similar to $A_2$ (see \eqref{A_2}) and the second one falls to the same category as $B_5, B_6$. For $B_7$ we have
	\begin{align*}
		&\int_{\Om}B_7\p v_1\p \hat{v}_2\hat{v}_0\,dzd\oz\\
		&= \int_{\Om}B_7\p(\gamma^{-1/2})\hat{v}_1\p \hat{v}_2\hat{v}_0\,dzd\oz + \int_{\Om}B_7\gamma^{-1/2}\p \hat{v}_1\p \hat{v}_2\,dzd\oz\\
		&= O(h^{-1/2+\e})\\
		&+ \int_{\Om}B_7\gamma^{-1/2}e^{(2\Phi-2\ol{\Phi})/h}\Big(\frac{1}{h}\p(\Psi + \Phi)(a_1+r_1) + \p a_1 + \p r_1\Big)\\
		&\times \Big(\frac{1}{h}\p(-\Psi + \Phi)(a_2+r_2) + \p a_2 + \p r_2\Big)(a_0+r_0)\,dzd\oz,
	\end{align*}
	where the first integral was similar to \eqref{A_2}. The rest we divide into three cases by the powers of $h$. If there is $h^0$, then using stationary phase or the estimates for $r_k$ (see \eqref{eq:CZ_rh_norms}, \eqref{eq: no critical point estimate}), these terms are $O(1)$. Similarly for the terms with $h^{-1}$ and for the terms with $h^{-2}$ with at least two $r_k$ we get $O(h^{-1/2+\e})$. Then we are left with the following:
	\begin{align*}
		&\int_{\Om}B_7\p v_1\p \hat{v}_2\hat{v}_0\,dzd\oz\\
		&= \frac{1}{h^2}\int_{\Om}B_7\gamma^{-1/2}e^{(2\Phi-2\ol{\Phi})/h}\p(\Psi + \Phi)\p(-\Psi + \Phi)\big(a_0a_1a_2 + a_1a_2r_0 + a_0a_2r_1\big)\,dzd\oz\\
		&+ O(h^{-1/2+\e}).
	\end{align*}
	For the last two terms we use the definitions of $r_0$ and $r_1$, as we did in \eqref{A_5_last}, and for the first one we use stationary phase to obtain for $C>0$
	\begin{equation}\label{B_7}
		\int_{\Om}B_7\p v_1\p \hat{v}_2\hat{v}_0\,dzd\oz = C\frac{1}{h}(B_7\gamma^{-1/2})|_{z=z_0} + O(h^{-1+\e}).
	\end{equation}
	This ends the proof.
\end{proof}

Next we consider the last term in $I_2$.

\begin{lem}\label{lemma_C}
	For $I_{2,3}$ we have $\lim_{h\to 0} hI_{2,3} = 0.$
\end{lem}

\begin{proof}
	For $I_{2,3}$ we get the following: 
	\begin{align*}
		-iI_{2,3}&=\sum_{k=1}^2\int_{\Om} i^{k-1}cT^kQ\p v_1 \hat{v}_0\hat{v}_2 + (-i)^{k-1}cT^kQ\dbar v_1 \hat{v}_0\hat{v}_2\,dzd\oz\\\notag
		&= \sum_{k=1}^2\int_{\Om} C_1\p v_1 \hat{v}_0\hat{v}_2 + C_2\dbar v_1 \hat{v}_0\hat{v}_2\,dzd\oz.
	\end{align*}
	These are at most $O(h^{-1/2+\e})$ by the same arguments as for the terms with $A_1, A_2$ and $A_3$ (see before and after equation \eqref{A_2}).
\end{proof}

Then we return to the integral $I_1$, that is to
\begin{equation*}
	I_1 = \sum_{k=1}^{2}\int_{\Om}T^k \p_kv_0g(\nabla v_1, \nabla v_2)\,dV_g.
\end{equation*}

\begin{lem}\label{lemma_D}
	For $I_1$ we have $\lim_{h\to 0} hI_1=0.$
\end{lem}

\begin{proof}
	We immediately use \eqref{holomorp_coord} to have
	\begin{align*}
		&\sum_{k=1}^{2} \int_{\Om} T^k\p_kv_0g(\nabla v_1, \nabla v_2)\,dV_g\\
		&= \sum_{k=1}^2 i \int_{\Om} T^k\left(i^{k-1}\p v_0 + (-i)^{k-1}\dbar v_0\right)\left(\p v_1 \dbar v_2 + \dbar v_1\p v_2\right)\,dzd\oz\\
		&= \sum_{k=1}^2 i \int_{\Om} D_1\p v_0 \p v_1\dbar v_2 + D_2\p v_0 \dbar v_1\p v_2 + D_3\dbar v_0 \p v_1\dbar v_2 + D_4\dbar v_0 \dbar v_1\p v_2\,dzd\oz.
	\end{align*}
	Here all $D_j$ again depend on the summation index $k$.
	
	Starting with the $D_1$ term we expand further (and as before we do not write the summation):
	\begin{align}\label{D_1_tilde}
		&\int_{\Om} D_1\p v_0 \p v_1\dbar v_2\,dzd\oz\\\notag
		&= \int_{\Om}\tilde D_1 \hat{v}_0\hat{v}_1\hat{v}_2 + \tilde D_1\hat{v}_0\hat{v}_1\dbar \hat{v}_2 + \tilde D_1\hat{v}_0\p \hat{v}_1 \hat{v}_2 + \tilde D_1\hat{v}_0\p \hat{v}_1\dbar \hat{v}_2 + \tilde D_1 \p \hat{v}_0\hat{v}_1\dbar \hat{v}_2 \\\notag
		&+ \tilde D_1\p \hat{v}_0\p \hat{v}_1 \hat{v}_2 + \tilde D_1 \p \hat{v}_0\hat{v}_1\hat{v}_2 + \tilde D_1\p \hat{v}_0 \p \hat{v}_1\dbar \hat{v}_2\,dzd\oz.
	\end{align}
	Here we have abused notation using $\tilde D_1$ for all the terms that do not depend on $h$. The first six terms are similar to $A_1, A_3, A_2, B_5, B_4$ and $B_3$, in that order, and they are $O(h^{-1/2+\e})$. For the first remaining term we have
	\begin{align*}
		&\int_{\Om}\tilde D_1 \p \hat{v}_0\hat{v}_1\hat{v}_2\,dzd\oz\\
		&= \int_{\Om}\tilde D_1 \p \hat{v}_0 \hat{v}_1\hat{v}_2\,dzd\oz\\
		&=-\frac{2}{h}\int_{\Om}\tilde D_1 \p\ol{\Phi} e^{-2\ol{\Phi}/h}(a_0+r_0)\hat{v}_1\hat{v}_2 + \tilde D_1e^{(2\Phi - 2\ol{\Phi})/h}\p r_0(a_1 + r_1)(a_2 + r_2)\,dzd\oz.
	\end{align*}
	The first term is $O(1)$ by stationary phase and the fact that $\p\ol{\Phi}(z_0)=0$ and the second term is $O(h^{1/2+\e})$ because of the term $\p r_0$ (see \eqref{eq:CZ_rh_norms}). For the last term in \eqref{D_1_tilde} we use that $\p a_0 =\dbar a_2 = \dbar(-\Psi + \Phi)=0=\dbar\ol{\Phi}(z_0)$ and stationary phase to have
	\begin{align*}
		&\int_{\Om}\tilde D_1\p \hat{v}_0\p \hat{v}_1\dbar \hat{v}_2\,dzd\oz\\
		&= \int_{\Om}e^{(2\Phi-2\ol{\Phi})/h}\p r_0\left(\frac{1}{h}\p(\Psi+\Phi)(a_1+r_1) + \p a_1 + \p r_1\right)\p r_2\,dzd\oz\\
		&= O(h^{1/2+\e}).
	\end{align*}
	In the last equality we used the estimates for $\p r_0, \p r_2$ (see \eqref{eq:CZ_rh_norms}, \eqref{eq: no critical point estimate}).
	
	The term with $D_2$ goes similarly as above with the term $D_1$ because $v_1$ and $v_2$ have holomorphic phases.
	
	Moving to the term with $D_3$ we have (using similar notation as with $D_1$)
	\begin{align*}
		&\int_{\Om}D_3\dbar v_0 \p v_1\dbar v_2\,dzd\oz\\
		&= \int_{\Om}\tilde D_3\hat{v}_0\hat{v}_1\hat{v}_2 + \tilde D_3\hat{v}_0\hat{v}_1\dbar \hat{v}_2 + \tilde D_3\hat{v}_0\p \hat{v}_1\hat{v}_2 + \tilde D_3\hat{v}_0\p \hat{v}_1\dbar \hat{v}_2 + \tilde D_3\dbar \hat{v}_0\hat{v}_1\dbar \hat{v}_2\\
		&+ \tilde D_3\dbar \hat{v}_0\hat{v}_1\hat{v}_2 + \tilde D_3\dbar \hat{v}_0\p \hat{v}_1\dbar \hat{v}_2\,dzd\oz.
	\end{align*}
	Again the first six terms are similar to ones before. These are $O(h^{-1/2+\e})$ by the arguments made for $A_1, A_3, A_2, B_5, B_6$ and $B_1$. The first remaining term is
	\begin{align*}
		&\int_{\Om}\tilde D_3\dbar \hat{v}_0 \hat{v}_1\hat{v}_2\,dzd\oz\\
		&= \int_{\Om}e^{(2\Phi-2\ol{\Phi})/h}\left(\left(-\frac{2}{h}\dbar \ol{\Phi}(a_0 + r_0) + \dbar a_0 + \dbar r_0\right)(a_1+r_1)(a_2+r_2)\right)\,dzd\oz\\
		&= O(1)
	\end{align*}
	by stationary phase and the estimates for $r_k$ (see \eqref{eq:CZ_rh_norms}, \eqref{eq: no critical point estimate}). The last remaining term is, by using again that the functions $a_2, \Phi$ and $\Psi$ are holomorphic,
	\begin{align*}
		&\int_{\Om}\tilde D_3 \dbar \hat{v}_0\p \hat{v}_1\dbar \hat{v}_2\,dzd\oz\\
		&= \int_{\Om}\tilde D_3 e^{(2\Phi -2\ol{\Phi})/h}\left(-\frac{2}{h}\dbar\ol{\Phi}(a_0+r_0) + \dbar a_0 + \dbar r_0\right)\\
		&\times\left(\frac{1}{h}\p(\Psi+\Phi)(a_1+r_1) + \p a_1 + \p r_1\right)\dbar r_2\,dzd\oz\\
		&= -\frac{2}{h^2}\int_{\Om}\tilde D_3 e^{(2\Phi -2\ol{\Phi})/h}\dbar\ol{\Phi}\p(\Psi+\Phi)a_1a_0\dbar r_2\,dzd\oz + O(h^{-1/2+\e})
	\end{align*}
	because of the estimates for $r_1, r_0$ and $\dbar r_2$ (see \eqref{eq:CZ_rh_norms}, \eqref{eq: no critical point estimate}). For the last term, we recall that (see \eqref{eq:rh_form} - \eqref{eq:sh_formula})
	\begin{align*}
		r_2 &= - \dbar^{-1}_{\Imm(-\Psi + \Phi)}s_2 = \dbar^{-1}e^{i\Imm(-\Psi + \Phi)/h}s_2\\
		s_2 &= -\op_{\Imm(-\Psi + \Phi)}^{*-1}(qa) - \sum_{j=1}^{\infty}T_2^j\op_{\Imm(-\Psi + \Phi)}^{*-1}(qa)
	\end{align*}
	and thus in holomorphic coordinates, denoting $\eta= \Imm(-\Psi + \Phi)$,
	\begin{align*}
		\dbar r_2 &= -e^{i\eta/h}s_2\\
		&= e^{i\eta/h}\p^{-1}(e^{i\eta/h}qa) + e^{i\eta/h}\sum_{j=1}^{\infty}T_2^j\p^{-1}(e^{i\eta/h}qa)\\
		&= e^{i\eta/h}\left(\frac{ih}{2}e^{i\eta/h}\frac{qa}{\p\eta} + \frac{ih}{2}\p^{-1}\left(e^{i\eta/h}\p\left(\frac{qa}{\p\eta}\right)\right)\right) + e^{i\eta/h}\sum_{j=1}^{\infty}T_2^j\p^{-1}(e^{\eta/h}qa).
	\end{align*}
    Here we used the expansion \eqref{eq: no critical point expansion II}.
	Now
	\begin{align*}
		-\frac{2}{h^2}\int_{\Om}\tilde D_3 e^{(2\Phi -2\ol{\Phi})/h}\dbar\ol{\Phi}\p(\Psi+\Phi)a_1a_0\frac{ih}{2}e^{2i\eta/h}\frac{qa}{\p\eta}\,dzd\oz = O(1)
	\end{align*}
	by stationary phase.
	Recall that 
    $$\p^{-1}\left(e^{i\eta/h}\p\left(\frac{qa}{\p\eta}\right)\right) = O(h^{1/2+\e}),\quad \p^{-1}(e^{i\eta/h}qa) = O(h)$$
    (by the estimate \eqref{eq:sobo_decayL2} and the expansion \eqref{eq: no critical point expansion II}) and $T_h = O(h^{1/2-\e})$. Then the remaining terms in $\int_{\Om}\tilde D_3 \dbar \hat{v}_0\p \hat{v}_1\dbar \hat{v}_2\,dzd\oz$ are $O(h^{3/2-\e})$
	
	Again the term with $D_4$ goes similarly as above with the term $D_3$ because $v_1$ and $v_2$ have holomorphic phases.
\end{proof}
Recall the definition of $T$ in \eqref{T_def} and that the arguments in Lemmas \ref{lemma_A}, \ref{lemma_B}, \ref{lemma_C} can be applied to the integral $I_3$ with the same results. Then we have the following for $T$:
\begin{lem}\label{T_PDE_1}
	The components $T^1, T^2$ of $T$ satisfy the elliptic equation
	\begin{equation}\label{T^k_equation}
		\p\dbar(T^1+iT^2) + E\p(T^1 + iT^2) + \p E(T^1 + iT^2) = 0\quad \text{in  }\Om
	\end{equation}
	for a smooth function $E$.
\end{lem}
\begin{proof}
	Multiplying $I-\tilde I$ by $h$ and taking the limit as $h\to 0$ we get after using Lemmas \ref{lemma_A}, \ref{lemma_B}, \ref{lemma_C}, \ref{lemma_D}
	\begin{equation*}
		\sum_{k=1}^2 2Ci^{k-1}(A_5\gamma^{-1})|_{z=z_0} + i\tilde C(B_7\gamma^{-1/2})|_{z=z_0} = 0.
	\end{equation*}
	Recall that $A_5=2T^kc^{-1}\dbar(\gamma^{-1/2}) + \dbar T^k\gamma^{-1/2}$ and $B_7=i^{k-1}T^k\dbar(\gamma^{-1/2})$.
	Doing the above argument for a dense set of points implies that $T_1, T_2$ satisfy
	\begin{align*}
		\sum_{k=1}^2 i^{k-1}\dbar T^k + \frac{i^{k-1}}{2C}\gamma^{3/2}\left(4Cc^{-1}\dbar(\gamma^{-1/2})\gamma^{-1} + i\tilde C \dbar(\gamma^{-1/2})\gamma^{-1/2}\right)T^k = 0.
	\end{align*}
	Letting $E:= \frac{1}{2C}\gamma^{3/2}\left(4Cc^{-1}\dbar(\gamma^{-1/2})\gamma^{-1} + i\tilde C \dbar(\gamma^{-1/2})\gamma^{-1/2}\right)$ we have
	\begin{align*}
		\dbar(T^1 + iT^2)+E(T^1 + iT^2) = 0.
	\end{align*}
	Operating with $\dbar$ on both sides gives \eqref{T^k_equation}.
\end{proof}

Next we choose solutions $v_k = \gamma^{-1/2}\hat{v}_k$, $k=0,1,2$, such that
\begin{align*}
	\hat{v}_0 &= e^{2\Phi/h}(\ol{a}_0 + r_0)\\
	\hat{v}_1 &= e^{-(\ol{\Psi}+\ol{\Phi})/h}(\ol{a}_1 + r_1)\\
	\hat{v}_2 &= e^{-(-\ol{\Psi}+\ol{\Phi})/h}(\ol{a}_2 + r_2)
\end{align*}
where $\Phi, \Psi, a_k$, $k=0,1,2$, are as before. Then we look at the integral $I-\tilde I$ with these solutions. The analysis is in most parts the same. The differences come from the terms $I_{2,1}$ and $I_{2,2}$ (and of course $I_{3,1}$ and $I_{3,2}$ but they are identical to $I_{2,1}$ and $I_{2,2}$). With the solutions above we get
\begin{align*}
	\lim_{h\to 0} hI_{2,1} &= \sum_{k=1}^2C(-i)^{k-1}(A_4\gamma^{-1})|_{z=z_0}\\
	\lim_{h\to 0} hI_{2,2} &= \sum_{k=1}^2i\tilde C(B_6\gamma^{-1/2})|_{z=z_0},
\end{align*}
where
\begin{align*}
	A_4 = \Big(2T^kc^{-1}\dbar(\gamma^{-1/2}) + \dbar T^k\gamma^{-1/2}\Big),\quad B_6 = (-i)^{k-1}T^k\p(\gamma^{-1/2}).
\end{align*}
Then, similarly as for Lemma \ref{T_PDE_1} we have the following:
\begin{lem}\label{T_PDE_2}
	The components $T^1, T^2$ of $T$ satisfy the elliptic equation
	\begin{equation*}
		\p\dbar(T^1-iT^2) + \tilde E\p(T^1 - iT^2) + \p \tilde E(T^1 - iT^2) = 0\quad \text{in  }\Om
	\end{equation*}
	for a smooth function $\tilde E$.
\end{lem}
Since we assume that $T$ is known on the boundary to high order, we get by combining Lemma \ref{T_PDE_1} and Lemma \ref{T_PDE_2} with unique continuation that $T^1=T^2=0$.

\vskip0.5cm

\noindent\textbf{Acknowledgment.} The authors would like to thank Niko Jokela and Gyula Csat\'o for helpful discussion. 
T.L. was supported by the Academy of Finland (Centre of Excellence in Inverse Modeling and Imaging, grant numbers 284715 and 309963). J.N. was supported by the Finnish Centre of Excellence in Inverse Modelling and Imaging (Academy of Finland Grant 284715), the Research Council of Finland (Flagship of Advanced Mathematics for Sensing Imaging and Modelling grants 359208 and 359183) and by the Emil Aaltonen Foundation.

\bibliographystyle{alpha}
\bibliography{ref}

\end{document}